\documentclass[12pt, reqno]{amsart}
\usepackage{amsmath,amsfonts,amssymb,amsthm,enumerate,appendix,caption}
\usepackage{cmap,mathtools}
\usepackage{paralist,bm}
\usepackage{graphics} 
\usepackage{epsfig} 
\usepackage{graphicx}
\usepackage{braket}
\usepackage{epstopdf}
 \usepackage[colorlinks=true]{hyperref}
\usepackage{xcolor}
\definecolor{myblue}{rgb}{0.0, 0.0, 1.0}
\definecolor{mygreen}{rgb}{0.01,0.75,0.20}
\hypersetup{ linkcolor=myblue, colorlinks=true, urlcolor=black, citecolor=red}
\usepackage{hyperref}
\usepackage[hyperpageref]{backref}

\newtheorem{theorem}{Theorem}[section]

\newtheorem{lemma}[theorem]{Lemma}
\newtheorem{proposition}[theorem]{Proposition}
\newtheorem{example}[theorem]{Example}

\newtheorem{definition}[theorem]{Definition}

\newtheorem{remark}[theorem]{Remark}

\newtheorem{remarks}[theorem]{Remarks}

\theoremstyle{definition}

\usepackage[margin=1in]{geometry}

\numberwithin{equation}{section}
\newcommand{\diver}{\mathrm{div}\,}
\newcommand{\dx}{\,\mathrm{d}x}

\newcommand{\dt}{\,\mathrm{d}t}

\newcommand{\dnu}{\,\mathrm{d}\nu}
\newcommand{\dmu}{\,\mathrm{d}\mu}

\newcommand{\dGamma}{\,\mathrm{d}\Gamma}
\newcommand{\dtildeGamma}{\,\mathrm{d}\tilde\Gamma}
\newcommand{\dtildeGa}{\,\mathrm{d}\tilde\Gamma}
            \def\gg{\gamma}
       \def\gd{\delta}      
\def \gth{\theta}                         \def\vge{\varepsilon}
           
            \def\gl{\lambda}
                 
    \def\gr{\rho}        
       
      \def\gw{\omega}
                
     \def\Gd{\Delta}      

    \def\Gs{\Sigma}      
\def\Gw{\Omega}              

\DeclarePairedDelimiter\norm{\lVert}{\rVert}%
\makeatletter
\let\oldnorm\norm
\def\norm{\@ifstar{\oldnorm}{\oldnorm*}}

\newcommand{\al} {\alpha}

\newcommand{\De} {\Delta}

\newcommand{\Ga} {\Gamma}
\newcommand{\om} {\omega}
\newcommand{\Om} {\Omega}

\newcommand{\noi} {\noindent}

\newcommand{\ra} {\rightarrow}

\newcommand{\wra} {\rightharpoonup}

\newcommand{\DpA}{\D_A^{1,p}(\Om)}
\newcommand{\DpAV}{\D_{A,V^+}^{1,p}(\Om)}

\newcommand{\wrastar} {\overset{\ast}{\rightharpoonup}}

\newcommand\restr[2]{{
  \left.\kern-\nulldelimiterspace 
  #1 
  \right|_{#2} 
  }}

\def\w{{\widetilde w}}

\def\w2{{W^{1,2}_0(\Om)}}
\def\hh2{{H^1_0(\Om)}}

\def\C{{\mathcal C}}
\def\D{{\mathcal D}}

\def\N{{\mathbb N}}
\def\F{{\mathcal F}}
\def\M{{\mathcal M}}
\def\S{\mathbb{S}}

\def\cp{{\rm Cap}_{\bf{u}}}
\def\cpsi{{\rm Cap}_{\bf{\psi}}}

\def\R{{\mathbb R}}

\def\({{\Big(}}
\def\){{\Big)}}

\def\ws2{{\F_{\frac{N}{2}}}}

\def\c1{{\C_c^1}}

\def\d{{\rm d}}

\def\dt{{\rm d}t}

\def\dx{{\rm d}x}

\def\H{{\mathcal{H}}}

\newcommand\supp{\mathrm{supp}\,}

\newcommand{\Hmm}[1]{\leavevmode{\marginpar{\tiny%
			$\hbox to 0mm{\hspace*{-0.5mm}$\leftarrow$\hss}%
			\vcenter{\vrule depth 0.1mm height 0.1mm width \the\marginparwidth}%
			\hbox to
			0mm{\hss$\rightarrow$\hspace*{-0.5mm}}$\\\relax\raggedright #1}}}
\newcommand{\loc}{{\rm loc}}

\begin{document}
	\title[The space of Hardy-weights]{The space of Hardy-weights for quasilinear equations: Maz'ya-type characterization and sufficient conditions for existence of minimizers}
	
	\author {Ujjal Das}
	
	\address {Ujjal Das, Department of Mathematics, Technion - Israel Institute of
		Technology,   Haifa, Israel}
	
	\email {ujjaldas@campus.technion.ac.il}
	
	\author{Yehuda Pinchover}
	
	\address{Yehuda Pinchover,
		Department of Mathematics, Technion - Israel Institute of
		Technology,   Haifa, Israel}
	
	\email{pincho@technion.ac.il}
	
	\begin{abstract}
		Let $p \in (1,\infty)$ and  $\Gw\subset \R^N$ be a domain. Let  $A: =(a_{ij}) \in  L^{\infty}_{\loc}(\Om;\R^{N\times N})$ be  a symmetric and locally uniformly
		positive definite matrix. Set 
		$|\xi|_A^2:= \displaystyle \sum_{i,j=1}^N a_{ij}(x) \xi_i \xi_j$, $\xi \in \R^N$, and let $V$ be a given potential in a certain local Morrey space. 
		We assume that the energy functional  
		$$Q_{p,A,V}(\phi):=\displaystyle\int_{\Om} [|\nabla \phi|_A^p + V|\phi|^p] \dx $$
		is nonnegative in $W^{1,p}(\Om)\cap C_c(\Om)$.
		
		We introduce a generalized notion of $Q_{p,A,V}$-capacity and characterize  the space of all Hardy-weights for the functional $Q_{p,A,V}$, extending Maz'ya's well known characterization of the space of Hardy-weights for the $p$-Laplacian. In addition, we provide various sufficient conditions on the potential $V$ and the Hardy-weight $g$ such that the best constant of the corresponding variational  problem is attained in an appropriate Beppo-Levi space.

		\medskip
		
		\noindent  2000  \! {\em Mathematics  Subject  Classification.}
		Primary  \! 49J40; Secondary  31C45, , 35B09, 35J62.\\[1mm]
	\end{abstract}
\maketitle
\section{Introduction} 
 Let $\Om \subset \R^N$ be a domain, and $p \in (1,\infty)$. Let  $A: =(a_{ij}) \in  L^{\infty}_{\loc}(\Om;\R^{N\times N})$ be  a symmetric and locally uniformly
positive definite matrix. Set  
$$|\xi|_A^2:= \, \braket{A(x)\xi,\xi}\, = \displaystyle \sum_{i,j=1}^N a_{ij}(x) \xi_i \xi_j \,,  \qquad  \ x\in \Om, \, \, \xi=(\xi_1,\ldots ,\xi_N) \in \R^N,$$
and let $\De_{p,A} u:= \diver(|\nabla u|^{p-2}_A A(x)\nabla u)$ be the $(p,A)$-Laplacian of $u$ (see \cite{HKM,Yehuda_Georgios}).

For a potential $V$ in $\mathcal{M}^{q}_{\loc}(p;\Om)$,  a certain local Morrey space (see Definition~\ref{def-morrey}), we consider the following quasilinear elliptic equation:
\begin{align} \label{PDE}
	Q'_{p,A,V}[u]:=-\De_{p,A} u + V|u|^{p-2}u=0 \qquad \mbox{in} \ \Om,
\end{align}
together with its energy functional
$$Q_{p,A,V}(\phi):=\displaystyle\int_{\Om} [|\nabla \phi|_A^p + V|\phi|^p] \dx \qquad \phi \in  W^{1,p}(\Om)\cap C_c(\Om).$$ 
{\bf Throughout the paper we assume that $Q_{p,A,V}\geq 0$ on $W^{1,p}(\Om)\cap C_c(\Om)$.}

We recall that, by the Agmon-Allegretto-Piepenbrink-type theorem \cite[Theorem~4.3]{Yehuda_Georgios}, $Q_{p,A,V}\geq 0$ on $W^{1,p}(\Om)\cap C_c(\Om)$ if and only if  \eqref{PDE} admits a weak positive solution (or positive supersolution) in $W^{1,p}_\loc(\Om)$.

\medskip

The first aim of the present paper is to characterize the space of all functions  $g \in L^1_{\loc}(\Om)$ such that the following Hardy-type inequality holds:
\begin{align} \label{WH}
\int_{\Om} |g| |\phi|^p \dx \leq C Q_{p,A,V}(\phi) \qquad  \forall \phi \in W^{1,p}(\Om)\cap C_c(\Om)   
\end{align}
for some $C>0$.  
A function $g$ satisfying  \eqref{WH} is called a {\em Hardy-weight} of $Q_{p,A,V}$ in $\Gw$. We denote the space of all Hardy-weights by 
$$\mathcal{H}_p(\Om,V):=\{g \in L^1_{\loc}(\Om) \mid  g  \text{ satisfies }  \eqref{WH}\}.$$
  If $\mathcal{H}_p(\Om,V)=\{0\}$, then  $Q_{p,A,V}$ is said to be {\em critical in} $\Gw$, otherwise,  $Q_{p,A,V}$ is {\em subcritical in} $\Gw$ \cite{Yehuda_Georgios}. If $Q_{p,A,V} \not \geq 0$ on  $W^{1,p}(\Om)\cap C_c(\Om)$, then $Q_{p,A,V}$ is said to be {\em supercritical in} $\Gw$.  
Recall that for $p \!\in \!(1,N)$, $\Gw\!=\! \R^N \! \setminus \! \{0\}$, $V\! = \! 0$, $A\!=\! I_{N \times N},$ and  $g(x)={C(p,N)}/{|x|^p}$, inequality \eqref{WH} corresponds to the classical Hardy inequality.

In the context of improving the classical Hardy inequality many examples of  Hardy-weights were produced,
  see \cite{Adimurthy_Mythily,BV,DFP,DP,Filippas} and the references therein. For $p=2$, $A=I_{N \times N}$, and bounded domain $\Om$,    it is well known that $L^r(\Om)\subset \mathcal{H}_{2}(\Om,{\bf{0}})$ with $r>\frac{N}{2}$  \cite{Manes-Micheletti}, $r=\frac{N}{2}$ \cite{Allegretto}. 
  Moreover, using the Lorentz-Sobolev embedding,  Visciglia \cite{Visciglia} showed that the Lorentz space $L^{\frac{N}{p},\infty}(\Om)\subset \mathcal{H}_{p}(\Om,{\bf{0}})$ when  $p\in (1,\infty)$, $A=I_{N \times N}$, and  $\Om$ is a general domain. In \cite[Theorem 8.5]{Mazya2}, Maz'ya gave an intrinsic characterization of a Hardy-weight (for the $p$-Laplacian) using the notion of $p$-capacity. In the recent paper \cite{AU}, using Maz'ya's characterization, the authors  introduced a norm on $\mathcal{H}_p(\Om,{\bf{0}})$, the space of all Hardy-weights for the $p$-Laplacian, making $\mathcal{H}_p(\Om,{\bf{0}})$ a Banach function space (see Definition~\ref{BFC}). 
  
  Inspired by \cite{Bianchini}, we extend the classical definition of $p$-capacity on compact sets in $\Gw$ to the case of the nonnegative 	functional $Q_{p,A,V}$ (see a detailed discussion in Subsection~\ref{subsec-cap}). 
  \begin{definition}[$Q_{p,A,V}$-capacity] \label{Def:Cap} {\em Let ${\bf{u}} \in W^{1,p}_{\loc}(\Om)\cap C(\Om)$ be a positive function.  
  		For a compact set $F \Subset \Om$, the {\em $Q_{p,A,V}$-capacity of $F$ with respect to $({\bf{u}},\Om)$} is defined by
  		$${\rm{Cap}}_{\bf{u}}(F,\Om):=\inf \{Q_{p,A,V}(\phi) \mid\phi \in \mathcal{N}_{F,\bf{u}}(\Om)\} \,,$$
  		where $\mathcal{N}_{F,\bf{u}}(\Om):=\{\phi \in  W^{1,p}(\Om)\cap C_c(\Om) \mid \phi \geq {\bf{u}} \ \text{on} \ F\}$. 
  	}
  \end{definition}
   Now we are in a position to extend  to our setting, the definition of a norm on the space of Hardy-weights. Let ${\bf{u}} \in W^{1,p}_{\loc}(\Om)\cap C(\Om)$ be a positive solution of \eqref{PDE}. Then, we equip $\mathcal{H}_p(\Om,V)$ with the norm:
\begin{eqnarray*}
	\norm{g}_{\H_p(\Om,V)} := \sup\left\{ \frac{\int_{F} |g||{\bf{u}}|^p \dx}{\cp(F,\Om)} \, \, \biggm| \,  F \Subset \Om \mbox{ is compact set s.t. } \cp(F,\Om)\ne 0 \right\},              
 \end{eqnarray*}
which is well defined due to the following extension of Maz'ya's well known characterization of the space of Hardy-weights for the $p$-Laplacian \cite[Theorem 8.5]{Mazya2}.
\begin{theorem} \label{Thm:Char}
Let $p \in (1,\infty)$ and  $g \in L^1_{\loc}(\Om)$. Then $\norm{g}_{\H_p(\Om,V)}<\infty$ if and only if the Hardy-type inequality \eqref{WH} holds.  Moreover, let $\mathcal{B}_g(\Om,V)$ be the best constant in \eqref{WH}, then
$$\norm{g}_{\H_p(\Om,V)} \leq \mathcal{B}_g(\Om,V) \leq C_H \norm{g}_{\H_p(\Om,V)},$$
where $C_H$ is  independent of $g$. Furthermore, $\|g\|_{\mathcal{B}(\Om,V)} :=\mathcal{B}_g(\Om,V)$ is an equivalent norm on $\mathcal{H}_p(\Om,V)$. In particular, up to the equivalence relation of norms, the norm $\norm{\cdot}_{\H_p(\Om,V)}$ is independent of the positive solution ${\bf{u}}$.
\end{theorem}
 We would like to remark that the main difficulty in proving Theorem~\ref{Thm:Char} arises due to the sign-changing behaviour of the potential  $V$. We overcome this difficulty by using the $Q_{p,A,V}$-capacity and the simplified energy functional (\cite{PinRa,PTT}, see Definition~\ref{Def:SEF}). In view of Theorem~\ref{Thm:Char}, we identify the space of Hardy-weights as 
$$\mathcal{H}_p(\Om,V)=\left\{g \in L^1_{\loc}(\Om) \mid \|g\|_{\H_p(\Om,V)}<\infty \right\}.$$  
In fact,  $\mathcal{H}_p(\Om,V)$ is a Banach function space (see Definition~\ref{BFC}), 
and in particular, $\mathcal{H}_p(\Om,V)$ is a Banach space. Moreover, under some conditions we show that certain weighted Lebesgue spaces are embedded in $\mathcal{H}_p(\Om,V)$ (see Theorem~\ref{Thm:Weighted}). 
\begin{definition}[Beppo-Levi space]\label{def_BL}\rm 
	The {\em generalized Beppo-Levi space} $\mathcal{D}_{A,V^+}^{1,p}(\Om)$ is the  completion of $W^{1,p}(\Om)\cap C_c(\Om)$ with respect to the norm 
		$$\norm{\phi}_{\mathcal{D}_{A,V^+}^{1,p}(\Om)} :=  \left[\||\nabla \phi|_A\|_{L^p(\Om)}^p+\|\phi\|_{_{L^p(\Om,V^+\dx)}}^p \right]^{1/p} \,.$$ 
\end{definition}
Recall that for $g\in \mathcal{H}_p(\Om,V)$,  $\mathcal{B}_g(\Om,V)$ is the best constant for the inequality  \eqref{WH}, i.e.,  
$$\frac{1}{\mathcal{B}_g(\Om,V)}= \S_g(\Om,V):=\inf \{Q_{p,A,V}(\phi) \bigm| \phi \in W^{1,p}(\Om) \cap C_c(\Om) \,, \int_{\Om} |g| |\phi|^p \dx =1\}.$$
We say that  {\em the best constant $\mathcal{B}_g(\Om,V)$ is attained} if $\S_g(\Om,V)$ is attained in $\mathcal{D}_{A,V^+}^{1,p}(\Om)$.
\begin{remark}
{\em Generally speaking, for $p\neq2$, the nonnegative functional $Q_{p,A,V}^{1/p}$ does not define a norm on $W^{1,p}(\Om)\cap C_c(\Om)$ unless $V\geq 0$ (see the discussion in \cite[Section~6]{PT07}). Therefore, to a subcritical functional $Q_{p,A,V^+}$, we associate the  Banach space $\mathcal{D}_{A,V^+}^{1,p}(\Om)$ which is  a complete, separable and reflexive  Banach space of functions (see Remark \ref{Remark_ref}).
We note that  the space $\mathcal{D}_{A,V}^{1,p}(\Om)$  for $V$ satisfying $\inf V>0$ appears in the literature (see for example \cite{Bartolo} and references therein),  in particular for $p=2$, the space $\mathcal{D}_{A,V}^{1,2}(\Om)$ when $\inf V>0$ appears in spectral theory of Schr\"odinger operators (see for example \cite{Benci}). 
} 
\end{remark} 
In the present paper, we also investigate the attainment of  $\S_g(\Om,V)$ in $\mathcal{D}_{A,V^+}^{1,p}(\Om)$. In this context, using the standard variational methods, it is easily seen that if $V^-=0$ and the map
$$T_g(\phi) := \int_\Gw |g| |\phi|^p \dx$$
is compact on the Beppo-Levi space $\mathcal{D}_{A,V^+}^{1,p}(\Om)$, then $\S_g(\Om,V)$ is attained in $\mathcal{D}_{A,V^+}^{1,p}(\Om)$. The second aim of the paper is to characterize the space of all $g \in \mathcal{H}_p(\Om,V)$ for which $T_g$
is compact on $\mathcal{D}_{A,V^+}^{1,p}(\Om)$. As in \cite[Section 2.4.2, pp. 130]{Mazya_book}, \cite[Theorem 8]{AU}, we give a necessary and sufficient condition on $g$ for which $T_g$ is compact on $\mathcal{D}_{A,V^+}^{1,p}(\Om)$, see Theorem \ref{Cpct_Char}. Moreover, we identify a subspace of $\mathcal{H}_p(\Om,V)$ that ensure the compactness of $T_g$.
Let
 $$\mathcal{H}_{p,0}(\Om,V):=\overline{\mathcal{H}_{p}(\Om,V) \cap L_c^{\infty}(\Om)}^{\|.\|_{\mathcal{H}_{p}(\Om,V)}}.$$
For $A=I_{N\times N}$ and $V=0$, it has been shown in \cite[Theorem 8]{AU} that  if $g \in \mathcal{H}_{p,0}(\Om,{\bf{0}})$, then $T_g$ is compact in $\mathcal{D}_{I,0}^{1,p}(\Om)$. In the present paper, we prove the analogous result for a symmetric, and locally uniformly positive definite matrix  $A: =(a_{ij}) \in  L^{\infty}_{\loc}(\Om;\R^{N\times N})$, and a potential $V$, see Theorem \ref{Thm:compactness} and Remark~\ref{Rmk:Hptilde}-$(i)$.
In fact, under some further assumptions on  the matrix $A$, the converse is also true for $p\in (1,N)$ (see Remark \ref{Rmk:Hptilde}-$(ii)$). 
 
\medskip

It is worth mentioning that $\S_g(\Om,V)$ might be achieved in $\mathcal{D}_{A,V^+}^{1,p}(\Om)$ without $T_g$ being compact. This leads to our next aim, which is to provide weaker sufficient conditions on $V$ and $g$ such that $\S_g(\Om,V)$ is attained by a function in $\mathcal{D}_{A,V^+}^{1,p}(\Om)$. For $A=I_{N \times N}$ and $V = 0$, such problems have been studied in \cite{AU,Smets,Tertikas} using the well known concentration compactness arguments due to P.~L.~Lions \cite{Lions1a, Lions2a}. In the present article, we give two different sufficient conditions such that $\S_g(\Om,V)$ is attained by a function in $\mathcal{D}_{A,V^+}^{1,p}(\Om)$ (see theorems~\ref{Thm:best_constant} and \ref{bestconst_thm}). In each theorem we assume that $Q_{p,A,V}$ admits a (different) spectral gap-type condition, but, while the proof of Theorem~\ref{Thm:best_constant} is based on criticality theory, the proof of Theorem~\ref{bestconst_thm} uses the standard concentration compactness method.

\begin{remarks}{\em 
(i) If the operator $Q_{p,A,V-\S_g(\Om)|g|}$ is subcritical in a domain $\Gw$, then the best Hardy constant in \eqref{WH} is not attained \cite{Yehuda_Georgios}.  

(ii) If $g\in \mathcal{H}_p(\Om,V)$ is an optimal Hardy-weights (i.e.,$Q_{p,A,V-\S_g(\Om)|g|}$ is null-critical with respect to $g$ in $\Gw$) \cite{DFP,DP}, then the best Hardy constant in \eqref{WH} is not attained.}
\end{remarks}

\section{Preliminaries}\label{sec-prelim}
\subsection{Notation}\label{sub-notation}
Throughout the paper, we use the following notation and conventions:
\begin{itemize}
	\item For $R\!>\!0$ and $x\! \in\! \R^n$, we denote by $B_R(x)$ the open ball of radius $R$ centered at $x$.
	\item $\chi_S$ denotes the characteristic function of a set $S\subset \R^n$.	
	\item We write $A_1 \Subset A_2$ if  $\overline{A_1}$ is a compact set, and $\overline{A_1}\subset A_2$.
	\item For any subset $A \subset \R^N$, we denote the interior of $A$ by $\mathring{A}$.
	\item $C$ refers to  a positive constant which may vary from  line to line.
	\item Let $g_1,g_2$ be two positive functions defined in $\gw$. We use the notation $g_1\asymp g_2$ in
	$\gw$ if there exists a positive constant $C$ such
	that
	$C^{-1}g_{2}(x)\leq g_{1}(x) \leq Cg_{2}(x)$ for all  $x\in \gw$.	
	\item For any real valued measurable function $u$ and $\omega\subset \R^n$, we denote
	$$\inf_{\omega}u:=\mathrm{ess}\inf_{\omega}u, \quad \sup_{\omega}u:=\mathrm{ess}\sup_{\omega}u, \quad  u^+:=\max(0,u), \quad u^-:=\max(0,-u).$$
	\item For a subspace $X(\Om)$ of measurable functions on $\Gw$, 
 $X_c(\Om) :=\{f \!\in \! X(\Om)\!\mid\! \supp f \!\Subset\! \Om\}$.  
\item Let $(X,\|.\|_X)$ be a normed space, and $Y \subset X$, $\overline{Y}^{\|\cdot\|_X}$ is the closure of $Y$ in $X$.
\item For a Banach space $V$ over $\R$, we denote by $V^*$ the space of continuous linear maps from  $V$ into $\R$.
	\item For any $1\leq p\leq \infty$, $p'$ is the H\"older conjugate exponent of $p$ satisfying $p'=p/(p-1)$.
\item For $1\leq p<n$,  $p^*:=np/(n-p)$ is the corresponding Sobolev critical exponent.
\end{itemize}

\subsection{Morrey spaces}\label{sub-morrey}
In this subsection we introduce the local Morrey spaces $\mathcal{M}^{q}_{\loc}(p;\Om)$ in which the potential $V$ belongs to, and briefly discuss some of their properties.
\begin{definition}[Morrey space]\label{def-morrey}
	{\em Let $q \in [1,\infty]$ and $\om \Subset \R^N$ be an open set.  For a measurable, real valued function $f$ defined in $\om$, we set
		$$\|f\|_{\mathcal{M}^q(\om)}=\displaystyle\sup \left\{ m_q(r)  \int_{\om \cap B_r(y)} |f| \dx \bigm| y\in \om, 0<r<\mathrm{diam}(\om) \right\} \,,$$
		where $m_q(r)= r^{-{N}/{q'}}$.
		We write $f \in \mathcal{M}^q_{\text{loc}}(\Om) $  if for any $\om \Subset \Om$,  we have $\|f\|_{\mathcal{M}^q(\om)}<\infty.$}
\end{definition}
\begin{remarks}\label{Rmk:Morrey} \rm
	$(i)$ Note that $\mathcal{M}^1_{\text{loc}}(\Om) = L^1_{\text{loc}}(\Om) $ and $\mathcal{M}^{\infty}_{\text{loc}}(\Om) = L^{\infty}_{\text{loc}}(\Om) $ (as vector spaces), but $ L^{q}_{\text{loc}}(\Om) \subsetneq  \mathcal{M}^{q}_{\text{loc}}(\Om) \subsetneq L^{1}_{\text{loc}}(\Om)$ for any $q \in (1,\infty)$.
	
	$(ii)$ For $f \in \mathcal{M}^q_{\text{loc}}(\Om)$ and $1\leq q<\infty$, it is easily  seen that for $B_\gr(x)  \subset  \Om$ we have $\|f\|_{B_r(x)}\leq \|f\|_{B_\gr(x)}$ for $r \leq \gr.$ 
\end{remarks}
Next we define a special local Morrey space $\mathcal{M}^{q}_{\text{loc}}(p;\Om)$ which depends on the underlying
exponent $1<p<\infty$.
\begin{definition}[Special Morrey space] \rm For $p \neq N$, we define 
	\begin{align*}
	\mathcal{M}^{q}_{\text{\loc}}(p;\Om):=  \begin{cases}
	\mathcal{M}^{q}_{\text{\loc}}(\Om)  \text{ with }  q>\frac{N}{p}  & \text{if} \ p<N,\\
	L^{1}_{\text{\loc}}(\Om) \   \quad & \text{if} \ p>N \,,
	\end{cases}    
	\end{align*}
	while for $p = N$, the Morrey space $\mathcal{M}^{q}_{\loc}(N;\Om)$ consists of all those $f$ such that for some $q > N$ and any $\om \Subset \Om$
	\begin{align*}
	\|f\|_{\mathcal{M}^q(N;\om)}=\displaystyle\sup \left\{m_q(r) \int_{\om \cap B_r(y)} |f| \dx
	\bigm|  y\in \om, r<\mathrm{diam}(\om) \right\} ,    
	\end{align*}
	where $m_q(r)=[\log(\frac{\mathrm{diam}(\om)}{r})]^{{q/N'}}$ for $0<r<\mathrm{diam}(\om)$ \cite[Theorem 1.94, and references therein]{Ziemer}.	For more details on Morrey spaces, see the monograph \cite{Ziemer} and references therein.
\end{definition}
We recall the Morrey-Adams inequality (see \cite[Theorem 4.1]{Ziemer}  and \cite[Theorem~7.4.1]{Pucci}).
\begin{proposition}[Morrey-Adams inequality] \label{Prop:MA} Let $p \in (1,\infty)$, $\om \Subset \R^N$ and $f \in \mathcal{M}^{q}_{\loc}(p;\om)$. Then 
	there exists a constant $C(N,p,q)>0$ such that for any $0<\delta<\gd_0$
	\begin{align} \label{ineq:MA}
	\int_{\om} |f| |\phi|^p \dx \leq \delta \int_{\om}  |\nabla \phi|^p \dx + \frac{C(N,p,q)}{\delta^{\frac{N}{pq-N}}} \|f\|_{\M^{q}(p; \om)}^{\frac{pq}{pq-N}} \int_{\om} |\phi|^p \dx \qquad \forall \phi \in W^{1,p}_0(\om).
	\end{align}
\end{proposition}
\subsection{Generalized Capacity} \label{subsec-cap} Let ${\bf{u}} \in W^{1,p}_{\loc}(\Om)\cap C(\Om)$ be a positive function. Recall our definition of ${\rm{Cap}}_{\bf{u}}(F,\Om)$, the  $Q_{p,A,V}$-capacity of a compact set 
$F \subset \Gw$ with respect to $({\bf{u}},\Om)$ (see Definition~\ref{Def:Cap}). In this subsection, we briefly discuss some of the  properties of ${\rm{Cap}}_{\bf{u}}(F,\Om)$ (cf. \cite{PinRa} and references therein).
\begin{remarks} \label{Cap_def_eqiv}{\em $(i)$  Let $\Om_1 \subset \Om_2$ be domains in $\R^N$. Then $\cp(\cdot, \Om_2) \leq \cp(\cdot, \Om_1).$ On the other hand, for two compact sets $F_1 \subset F_2$ in $\Om$, we have $\cp(F_1, \Om) \leq \cp(F_2, \Om).$

$(ii)$ Notice that, for any compact set $F \subset \Om$ we have
$${\rm{Cap}}_{\bf{u}}(F,\Om)=\inf \{Q_{p,A,V}(\psi  {\bf{u}}) \mid \psi {\bf{u}} \in W^{1,p}(\Om) \cap C_c(\Om) \ \text{with} \ \psi \geq 1 \ \text{on} \ F\} \,.$$

$(iii)$ One can easily verify that $\phi \in \mathcal{N}_{F,\bf{u}}(\Om)$ implies that $|\phi| \in \mathcal{N}_{F,\bf{u}}(\Om)$ and $Q_{p,A,V}(|\phi|)=Q_{p,A,V}(\phi)$. Hence, in the definition of the $Q_{p,A,V}$-capacity, it is enough to  consider only nonnegative test functions $\phi \in \mathcal{N}_{F,\bf{u}}(\Om)$. Furthermore, for any $\phi \in \mathcal{N}_{F,\bf{u}}(\Om)$,   define $\hat{\phi}=\min\{\phi,{\bf{u}}\}$. Then, $\hat{\phi} \in \mathcal{N}_{F,\bf{u}}(\Om)$ and  $Q_{p,A,V}(\hat{\phi}) \leq Q_{p,A,V}(\phi)$ (cf. the proof of  
\cite[Proposition~4.1]{Bianchini} and Lemma~\ref{lem-smaller_null-seq}). Thus, it is easy to see that
$${\rm{Cap}}_{\bf{u}}(F,\Om)=\inf \{Q_{p,A,V}(\phi) \mid   \phi \in \widetilde{\mathcal{N}}_{F,\bf{u}}(\Om)\} \,,$$
where $\widetilde{\mathcal{N}}_{F,\bf{u}}(\Om)=\{\phi \in W^{1,p}(\Om)\cap C_c(\Om) \mid  \phi = {\bf{u}} \ \text{on} \ F, 0 \leq \phi \leq {\bf{u}} \ \text{in} \ \Om\}.$

$(iv)$ Choosing ${\bf{u}} = 1$, we see that our definition of $Q_{p,A,V}$-capacity coincides with the classical definitions. For instance, see \cite{HKM, Mazya2}  for the case $V= 0$, $A = I_{N \times N}$, \cite{Pin_Tin} for $V \neq 0$, $A = I_{N \times N}$, and \cite{PinRa} for $V \neq 0$, 
$A \in L^{\infty}_{\loc}(\Om;\R^{N\times N})$.

$(v)$ The local uniform ellipticity of $A \in L^{\infty}_{\loc}(\Om;\R^{N\times N})$ ensures that there exists a positive measurable function $\theta: \Om \to (0,\infty)$ such that
\begin{align*} 
 \frac{1}{\theta(x)} |\xi|_{I} \leq |\xi|_{A} \leq \theta(x) |\xi|_{I}  \qquad \forall x \in \Om, \xi \in \R^N ,
 \end{align*} 
and  $\theta$ is called  a {\em local uniform ellipticity function}. Suppose that the local uniform ellipticity function  $\theta \in L^{\infty}(\Om)$, then  $1\leq \theta(x) \leq \|\theta\|_{L^\infty(\Gw)}$ in $\Gw$. Therefore, $A\in L^\infty(\Gw,\R^{N\times N})$, and $A$ is uniformly elliptic in $\Gw$.  
Furthermore, if  $p \in (1,N)$,  then the Gagliardo-Nirenberg-Sobolev inequality implies that for any compact $F \Subset \Om$ and any $\phi \in \mathcal{N}_{F,{\bf{1}}}(F,\Om)$
\begin{align*}
    |F|^{\frac{p}{p^*}} \leq \left[\int_{\Om} |\phi|^{p^*} \dx\right]^{{p}/{p^*}} \leq C \int_{\Om} |\nabla \phi|_I^p\dx \leq C \int_{\Om} |\theta(x)|^p |\nabla \phi|_A^p \dx \leq C \|\theta \|_{L^{\infty}(\Om)}^p \int_{\Om}  |\nabla \phi|_A^p \dx
\end{align*}
 for some $C>0$. Thus, if  $V=0$, $p \in (1,N)$, and $\theta \in L^{\infty}(\Om)$, then there exists a constant $C>0$ such that
$$|F|^{\frac{p}{p^*}} \leq C {\rm{Cap}}_{{\bf{1}}}(F,\Om) \qquad  \forall F \Subset \Om.$$

$(vi)$ Let $F\Subset \Gw$ be a compact set. We claim that $\cp(F,\Om)=0$ if and only if ${\rm{Cap}}_{{\bf{1}}}(F,\Om)=0$. Assume first that ${\rm{Cap}}_{{\bf{1}}}(F,\Om)=0$. Consider ${\bf{u}}_F={{\bf{u}}}/{\|{\bf{u}}|_F\|_{{\infty}}}$, where ${\bf{u}}|_F$ is the restriction of ${\bf{u}}$ on $F.$ Then, it is clear that ${\rm{Cap}}_{{\bf{u}}_F}(F,\Om)\leq {\rm{Cap}}_{{\bf{1}}}(F,\Om).$ This implies that ${\rm{Cap}}_{{\bf{u}}_F}(F,\Om)=0$. Since ${\rm{Cap}}_{{\bf{u}}_F}(F,\Om)=\|{\bf{u}}|_F\|_{{\infty}}^p \cp(F,\Om)$, it follows that $\cp(F,\Om)=0$. Conversely, suppose that $\cp(F,\Om)=0$, and consider $\widetilde{{\bf{u}}}_F= {{\bf{u}}}/{\inf_{F}{\bf{u}}}$. Then, following a similar argument as above, we conclude that ${\rm{Cap}}_{{\bf{1}}}(F,\Om)=0$.
}
\end{remarks}
\subsection{Simplified Energy Functional}
Recall our assumption that $Q_{p,A,V}\geq 0$ on $W^{1,p}(\Om) \cap C_c(\Om)$, and notice that a priori, the corresponding Lagrangian might take negative values. Fortunately, due to a Picone identity \cite{PinRa} see also \eqref{eq-picone}), this Lagrangian is equal to a nonnegative Lagrangian (depending on a positive solution $\bf{u}$ of \eqref{PDE}), but  the latter Lagrangian contains indefinite terms. Fortunately, $Q_{p,A,V}$ admits an equivalent functional  $E_{{\bf{u}}}$ depending on $\bf{u}$ which, roughly speaking, is easier to handle since it contains only nonnegative terms.  
\begin{definition}[\cite{PinRa,PTT}] \label{Def:SEF}
 {\em Let $p\in (1,\infty)$ and ${\bf{u}} \in W^{1,p}_{\loc}(\Om)\cap C(\Om)$ be a positive solution of \eqref{PDE}. Then the {\em simplified energy functional} $E_{{\bf{u}}}: W^{1,p}(\Om)\cap C_c(\Om) \to \R$ of the functional $Q_{p,A,V}$ in $\Gw$ is defined as
 $$E_{{\bf u}}(\phi):=\int_{\Om} {\bf u}^2 |\nabla \phi|_A^2(\phi|\nabla {\bf u}|_A +{\bf u} |\nabla \phi|_A)^{p-2}\dx \qquad  \phi \in W^{1,p}(\Om)\cap C_c(\Om).$$ 
}
\end{definition}
\begin{proposition}[{\cite[Lemma~2.2]{PTT} and \cite[Lemma 3.4]{PinRa}}] \label{Prop:simp_energy}
Let $p\in (1,\infty)$ and ${\bf{u}} \in W^{1,p}_{\loc}(\Om)\cap C(\Om)$ be a positive solution of \eqref{PDE}. Then  there exists $C_1,C_2>0$ such that
$$C_1 Q_{p,A,V}({\bf u}\phi) \leq  E_{{\bf u}}(\phi) \leq C_2 Q_{p,A,V}({\bf u}\phi) \qquad \forall {\bf u}\phi \in W^{1,p}(\Om)\cap C_c(\Om).$$
\end{proposition}
\begin{remark} {\rm If ${\bf{u}} \in W^{1,p}_{\loc}(\Om)\cap C(\Om)$ is a positive solution of \eqref{PDE}, one may equivalently define a $Q_{p,A,V}$-capacity using the simplified energy functional.}
\end{remark}
\subsection{Banach function space}
\begin{definition}\label{BFC}
   {\em A normed linear space $(X(\Om),\norm{.}_{X(\Om)})$ (the value
   	$\norm{f}_{X(\Om)} =\infty$ is admitted) of measurable functions on $\Gw$ is called a {\em Banach function space} if the following conditions are satisfied:
  \begin{enumerate}
   \item $\norm{f}_{X(\Om)} = \|  |f|  \|_{X(\Om)}$, for all $f \in X(\Om)$,
    \item if $ (f_n)$ is a nonnegative sequence of function in $X(\Om)$, increasing to $f$, then $\norm{f_n}_{X(\Om)}$ increases to $ \norm{f}_{X(\Om)}$.
  \end{enumerate} }
  \end{definition}
   The norm $\norm{.}_{X(\Om)}$ is called a {\em Banach function space norm} on $X(\Om)$ \cite[Section 30, Chapter 6]{Zaanen}. Indeed, the restriction of a Banach function space to  $\{f\in X(\Om) \mid  \norm{f}_{X(\Om)}<\infty \}$ is complete \cite[Theorem 2, Section 30, Chapter 6]{Zaanen}.
\section{Characterization of Hardy-weights}
The present section is devoted mainly to the proof of Theorem~\ref{Thm:Char}. As applications of Theorem~\ref{Thm:Char}, we derive a necessary integral condition as in \cite[Theorem 3.1]{KP} for $|g|$ being a Hardy-weight, and show that certain weighted Lebesgue spaces are embedded into $\H_p(\Om,V)$.
\begin{proof}[Proof of Theorem~\ref{Thm:Char}]
{\bf{Necessity:}} Let $g \in L^1_{\loc}(\Om)$ be such that the  Hardy-type inequality \eqref{WH} holds. Let  $F \subset \Om$ be a compact set. Then by \eqref{WH} we have for any measurable function   $\psi$ such that $\psi {\bf{u}} \in W^{1,p}(\Om) \cap C_c(\Om)$ with $\psi \geq 1$ on $F$
\begin{align*}
 \int_{F} |g| | {\bf{u}}|^p \dx \leq \int_{\Om} |g| |\psi {\bf{u}}|^p \dx \leq   \mathcal{B}_g (\Om,V) Q_{p,A,V}(\psi {\bf{u}}) \,.
\end{align*}
By taking infimum over all $\psi {\bf{u}}\in W^{1,p}(\Om) \cap C_c(\Om)$ with $\psi \geq 1$ on $F$, and using Remark \ref{Cap_def_eqiv}-$(ii)$ we obtain
\begin{align*} 
   \int_F |g| |{\bf{u}}|^p \dx \leq \mathcal{B}_g (\Om,V)  {\rm{Cap}}_{\bf{u}}(F,\Om) \qquad \mbox{for all compact sets  $F$ in $\Om$}.
\end{align*}
Hence, $\|g\|_{\H_p(\Om,V)} \leq \mathcal{B}_g (\Om,V).$

{\bf{Sufficiency:}} Assume that $\|g\|_{\H_p(\Om,V)}<\infty$, i.e., 
\begin{align} \label{Mazya_Cond}
   \int_F |g| |{\bf{u}}|^p \dx \leq \|g\|_{\H_p(\Om,V)} {\rm{Cap}}_{\bf{u}}(F,\Om) \qquad \mbox{for all compact sets  $F$ in $\Om$}.
\end{align}
  Consider the measure $\mu = |g| |{\bf{u}}|^p \dx.$ Then for any nonnegative $\psi \in W^{1,p}(\Om)\cap C_c(\Om)$ (see for example \cite[Theorem~1.9]{Ziemer}),
\begin{align} \label{first_est}
\int_{\Om} \! &|g| |\psi {\bf{u}}|^p  \dx \! = p \int_0^{\infty} \mu(\{\psi \geq t\}) t^{p-1}  \dt = p \sum_{j=-\infty}^{\infty}\int_{2^j}^{2^{j+1}} \mu(\{\psi \geq t\}) t^{p-1} \dt \nonumber \\
& \leq \!(2^p-1) \!\! \sum_{j=-\infty}^{\infty} \!\!\! 2^{pj} \mu(\{\psi \!\geq \! 2^j\}) 
 \leq \|g\|_{\H_p(\Om,V)} (2^p-1) \!\!\sum_{j=-\infty}^{\infty} \!\!\! 2^{pj}  {\rm{Cap}}_{\bf{u}}(\{\psi \! \geq \! 2^j\},\Om) \,.
\end{align}
\noi {\bf{Case 1.}} Let $p\geq 2$.
Consider the function $\psi_j \in W^{1,p}(\Om)\cap C_c(\Om)$
\begin{eqnarray*} 
\psi_j(x):=
\begin{cases}
\displaystyle 0 \ \ \ \ \ \ \ \ \  \ \ \ \   \mbox{if} \ \psi \leq 2^{j-1} \,, \\
 \frac{\psi}{2^{j-1}}-1  \,\,\, \ \ \mbox{if }  2^{j-1} \le  \psi \leq 2^j \,, \\
1 \,\,\,\,\,\,\,\,\,\,\,\,\,\ \ \ \ \ \  \mbox{if \ } 2^j \leq \psi
  \,.
\end{cases}
\end{eqnarray*}
Then
$ {\rm{Cap}}_{\bf{u}}(\{\psi \geq 2^j\},\Om) \leq Q_{p,A,V}(\psi_j {\bf{u}})$.
 Let $A_j \! :=\! \{x \!\in \! \Gw \mid 2^{j-1} \!\leq\!  \psi(x) \!\leq \!2^j\}$. The simplified energy  (Proposition~\ref{Prop:simp_energy}) implies 
\begin{align} \label{case1} 
& {\rm{Cap}}_{\bf{u}}(\{\psi \geq 2^j\},\Om) \leq  Q_{p,A,V}(\psi_j {\bf{u}}) 
 \leq C\int_{\Om} {\bf{u}}^2 |\nabla \psi_j|_A^2 \left[\psi_j |\nabla {\bf{u}}|_A + {\bf{u}} |\nabla \psi_j|_A\right]^{p-2}  \!\dx \nonumber \\  
& \!= \!C \!\!\int_{A_j}\!\!\!\! {\bf{u}}^2 \frac{|\nabla \psi|_A^2}{2^{2(j-1)}} \!\left[\! \left(\frac{\psi}{2^{j-1}} \!- 1\!\right)\! |\nabla {\bf{u}}|_A \!+\!  {\bf{u}}  \frac{|\nabla \psi|_A}{2^{j-1}} \!\right]^{p-2} \! \!\!\!\dx    
 \!\leq\! C\!\!\int_{A_j}\!\!\!\!{\bf{u}}^2 \frac{|\nabla \psi|_A^2}{2^{2(j-1)}}\! \left[\! \frac{\psi}{2^{j-1}} |\nabla {\bf{u}}|_A 
 \!+\! {\bf{u}}  \frac{|\nabla \psi|_A}{2^{j-1}} \!\right]^{p-2} \!\! \!\!\dx \nonumber \\
&=\left(\frac{C}{2^{p(j-1)}} \right) \int_{A_j} \left[{\bf{u}}^2 |\nabla \psi|_A^2 \right] \left[ \psi|\nabla {\bf{u}}|_A + {\bf{u}}  |\nabla \psi|_A \right]^{p-2}   \dx . 
\end{align}  
\noi {\bf{Case 2.}} Let $p< 2.$
Consider the function
\begin{eqnarray*} \label{crufn}
\widetilde{\psi}_j(x):=
\begin{cases}
\displaystyle 0 \ \ \ \ \ \ \ \ \  \ \ \ \ \ \ \ \  \  \mbox{if} \ \psi \leq 2^{j-1} \,, \\
 \left[\frac{\psi}{2^{j-1}}-1\right]^{{2}/{p}}  \,\, \ \ \mbox{if }  2^{j-1} \le  \psi \leq 2^j \,, \\
1 \,\,\,\,\,\,\,\,\,\,\,\,\,\ \ \ \ \ \ \ \ \ \ \ \ \mbox{if \ } 2^j \leq \psi.
\end{cases}
\end{eqnarray*}
Then $\widetilde{\psi}_j \in W^{1,p}(\Om)\cap C_c(\Om)$ and
$ {\rm{Cap}}_{\bf{u}}(\{\psi \geq 2^j\},\Om) \leq   Q_{p,A,V}(\widetilde{\psi}_j {\bf{u}})$, and similarly we obtain
\begin{align} \label{case2}
& \, {\rm{Cap}}_{\bf{u}}(\{\psi \geq 2^j\},\Om) \leq  Q_{p,A,V}(\widetilde{\psi}_j {\bf{u}}) 
  \leq C\int_{\Om} {\bf{u}}^2 |\nabla \widetilde{\psi}_j|_A^2 \left[\widetilde{\psi}_j |\nabla {\bf{u}}|_A + {\bf{u}} |\nabla \widetilde{\psi}_j|_A\right]^{p-2} \!\! \dx \nonumber \\  
& \!= \!\!\left(\!\frac{4C}{p^2} \!\right) \!\! \int_{A_j}\! \!\!\!{\bf{u}}^2 \frac{\left[\frac{\psi}{2^{j-1}} \! - \! 1\right]^{\frac{2(2-p)}{p}}\!
	|\nabla \psi|_A^2}{2^{2(j-1)}}\!\left[ \!\left[\frac{\psi}{2^{j-1}} \!-\!1\right]^{\frac{2}{p}} \!\!\!|\nabla {\bf{u}}|_A \!+\! {\bf{u}}  \frac{\left[\frac{\psi}{2^{j-1}}-1\right]^{\frac{(2-p)}{p}}|\nabla \psi|_A}{2^{j-1}} \!\right]^{p-2}  \!\!\! \!\dx  \nonumber \\  
& \!\leq\! \left(\!\frac{4C}{p^2}\! \right)\!\!  \int_{A_j} \!\!\!\!{\bf{u}}^2 \frac{\left[\frac{\psi}{2^{j-1}}\!- \! 1\right]^{\frac{2(2-p)}{p}}|\nabla \psi|_A^2}{2^{2(j-1)}} \!\! \left[ \!\left[\frac{\psi}{2^{j-1}} \!- \!1\right]^{\frac{2}{p}} |\nabla {\bf{u}}|_A + {\bf{u}}  \frac{\left[\frac{\psi}{2^{j-1}}-1\right]^{\frac{2}{p}}|\nabla \psi|_A}{2^{j-1}} \right]^{p-2}  \!\!\! \dx  \nonumber \\
& \!= \! \left(\!\frac{4C}{p^2} \!\right)\!\!  \int_{A_j}\! \!\! {\bf{u}}^2 \frac{|\nabla \psi|_A^2}{2^{2(j-1)}} \!  \left[  |\nabla {\bf{u}}|_A \! +\!  {\bf{u}}  \frac{|\nabla \psi|_A}{2^{j-1}} \right]^{p-2} \!\!\! \!  \dx 
 \!\leq  \!  \left(\! \frac{4C}{p^2} \! \right) \!\!  \int_{A_j}\! \!\! \!{\bf{u}}^2 \frac{|\nabla \psi|_A^2}{2^{2(j-1)}} \!\left[  \frac{\psi}{2^{j}}|\nabla {\bf{u}}|_A \! +\!  {\bf{u}}  \frac{|\nabla \psi|_A}{2^{j}} \right]^{p-2} \!\!  \! \! \dx   \nonumber \\
& = \left(\frac{16C}{2^{pj}p^2} \right) \int_{A_j} {\bf{u}}^2 |\nabla \psi|_A^2  \left[ \psi|\nabla {\bf{u}}|_A + {\bf{u}}  |\nabla \psi|_A \right]^{p-2}  \dx. 
\end{align} 
Subsequently, using \eqref{case1} (for $p \geq 2$) and \eqref{case2} (for $p<2$), we obtain from \eqref{first_est}  that
\begin{align} \label{second_est1}
&\int_{\Om} |g| |\psi {\bf{u}}|^p  \dx  \leq C_2\|g\|_{\H_p(\Om,V)} \sum_{j=-\infty}^{\infty} 2^{pj} \left(\frac{1}{2^{pj}} \right) 
\int_{A_j}{\bf{u}}^2 |\nabla \psi|_A^2 \ \left[ \psi|\nabla {\bf{u}}|_A + {\bf{u}}  |\nabla \psi|_A \right]^{p-2}  \dx \nonumber \\
&= C_2\|g\|_{\H_p(\Om,V)}\sum_{j=-\infty}^{\infty}  \int_{A_j} {\bf{u}}^2 |\nabla \psi|_A^2 \left[ \psi |\nabla {\bf{u}}|_A + {\bf{u}}|\nabla \psi|_A \right]^{p-2}   \dx \nonumber \\
& = C_2\|g\|_{\H_p(\Om,V)} \int_{\Om }  \left[{\bf{u}}^2 |\nabla \psi|_A^2\right] \left[ \psi |\nabla {\bf{u}}|_A + {\bf{u}}|\nabla \psi|_A \right]^{p-2}\!\! \dx 
\leq C_H\|g\|_{\H_p(\Om,V)} Q_{p,A,V}({\bf{u}} \psi) \,.
\end{align}
Therefore, for any nonnegative $\phi \in W^{1,p}(\Om)\cap C_c(\Om)$, by taking $\psi = {\phi}/{{\bf{u}}}$ in \eqref{second_est1}, we get
$$\int_{\Om} |g| |\phi|^p \ \dx \leq C_H\|g\|_{\H_p(\Om,V)} Q_{p,A,V}(\phi) \,.$$
Since $|\phi| \in W^{1,p}(\Om)\cap C_c(\Om)$ for $\phi \in W^{1,p}(\Om)\cap C_c(\Om)$, our claim follows.

Furthermore,  $\|g\|_{\mathcal{B}(\Om,V)} :=\mathcal{B}_g(\Om,V)$ also defines a norm on $\H_p(\Om)$.   From the above proof, it follows that this norm is equivalent to the norm $\| \cdot \|_{\H_p(\Om,V)}$.
Therefore, up to the equivalence relation of norms, the norm $\norm{\cdot}_{\H_p(\Om,V)}$ is independent of the positive solution ${\bf{u}}$.
\end{proof}
\begin{remarks} \label{Rmk:Prop_cap} {\em 
  $(i)$ Let $\Gw$ be a domain in $\R^N$. Theorem \ref{Thm:Char} implies that $\cp(F,\Om) = 0$ for any compact set $F \Subset \Om$ if and only if  $Q_{p,A,V}$ is critical in $\Gw$.
  In fact, by \cite[Theorem 6.8]{PinRa},  $Q_{p,A,V}$ is critical in $\Gw$ if and only if ${\rm{Cap}}_{{\bf{1}}}(F,\Om)=0$ for any compact set $F \Subset \Om$, and 
 by Remark \ref{Cap_def_eqiv}-$(vi)$, for any compact set $F \Subset \Om$, ${\rm{Cap}}_{{\bf{1}}}(F,\Om)=0$ if and only if $\cp(F,\Om)=0$. 

$(ii)$ In Theorem \ref{Thm:Char}, if we take $A=I_{N \times N}$, $V=0$, and ${\bf{u}} = 1$, we obtain Maz'ya's characterization of Hardy-weights for the $p$-Laplacian \cite[Theorem 8.5]{Mazya2}.

$(iii)$ Frequently, it is easier to compute or estimate $\| \cdot \|_{\mathcal{B}(\Om,V)}$ than $\| \cdot \|_{\H_p(\Om,V)}$.}
\end{remarks}
\subsection{Necessary condition for being a Hardy-weight}\label{KP-condition}
 In the recent paper \cite{KP}, the authors proved  a necessary condition for $|g|$ to be a Hardy-weight. Here we show that this condition can be derived directly from Theorem~\ref{Thm:Char}. First we recall the definition of positive solution of minimal growth at infinity.
\begin{definition}
	\emph{Let $K_{0}$ be a compact subset in  $\Omega$. A positive solution~$u$ of~$Q'_{p,A,V}[w]=0$ in~$\Omega\setminus K_{0}$, is called a \emph{positive solution of minimal growth in a neighborhood of infinity} in $\Omega$ if for any smooth compact subset~$K$ of~$\Omega$ with~$K_{0}\Subset \mathring{K}$, any positive supersolution~$v\in C\big(\Omega\setminus \mathring{K}\big)$ of ~$Q'_{p,\mathcal{A},V}[w]=0$ in~$\Omega\setminus K$ such that~$u\leq v$ on~$\partial K$, satisfies $u\leq v$ in~$\Omega\setminus K$. }
\end{definition} 
\begin{remark}\label{rem_min_gr} {\em 
If $Q_{p,A,V}$ is subcritical (critical) in a domain $\Gw$, then respectively, its minimal  positive Green function (ground state) is a positive solution of the equation $Q'_{p,A,V}[w]=0$ of minimal growth in a neighborhood of infinity in $\Omega$ \cite{Yehuda_Georgios}. For the definition and properties of a ground state see Definition~\ref{def-gs}, Remark~\ref{rem-gs} and \cite{Yehuda_Georgios}.}
	\end{remark}
Now we derive the necessary integral condition of \cite{KP} in the following proposition.
 \begin{proposition}\label{Thm:KP}\cite[Theorem 3.1]{KP} 
Let $\Gw$ be a domain and $K\Subset \Gw\subset \R^N$ with $\mathring{K} \!\neq \! \emptyset$.  Let ${\bf{u}} \in W^{1,p}_{\loc}(\Om)\cap C(\Om)$ be a positive solution  of $Q'_{p,A,V}[u]=0$ in $\Om \setminus K$ of minimal growth at a neighborhood of infinity in $\Gw$. If $g \in \mathcal{H}_p(\Om,V)$, then  $ \int_{\Om \setminus K_1}  |g| {\bf{u}}^p \dx< \infty$, where $K\Subset \mathring{K_1} \Subset \Gw$.
\end{proposition} 
\begin{proof} Let $g \in \mathcal{H}_p(\Om,V)$, and $K\Subset \Om$ such that $\mathring{K} \!\neq \! \emptyset$. By  \cite[Proposition 4.19]{Yehuda_Georgios}, there exists $0\leq V_K \!\in \! C_c^{\infty}(\Om)$ with support in $K$ such that $Q_{p,A,V-V_K}$ is critical  in $\Gw$. Let ${\bf \psi}$ be the ground state of $Q_{p,A,V-V_K}$ in $\Gw$. Thus, by  \cite{Arxiv_Pinchover,PinRa} and Remark~\ref{Rmk:Prop_cap}-$(i)$, 
	$\cpsi(F,\Om) \!=\! 0$ for any compact set $F \Subset \Om$ with nonempty interior. Let $\bigcup_n \Om_n= \Om$ be a compact smooth exhaustion of $\Om$. For each $n \in \N$, there exists $0 \!\leq \! \phi_n \in \widetilde{\mathcal{N}}_{\Om_n,{\bf \psi}}(\Om)$ such that $Q_{p,A,V-V_K}(\phi_n) < {1}/{n}$. Hence,
\begin{align*}
  \int_{\Om_n}  |g| {\bf \psi}^p\dx 
  \leq \int_{\Om_n } |g| \phi_n^p \dx 
  \leq C Q_{p,A,V}(\phi_n) & = C Q_{p,A,V-V_K}(\phi_n) + C \int_{\Om} V_K |\phi_n|^p \dx\\
    & \leq \frac{C}{n} +C \|V_K\|_{L^{\infty}(\Om)} \int_{K}  {\bf \psi}^p \dx
\end{align*}
Using Fatou's lemma, we conclude that $ \int_{\Om}  |g| {\bf \psi}^p \dx< \infty$. Hence, it follows from Remark \ref{rem_min_gr} that  $ \int_{\Om\setminus K_1}  |g| {\bf u}^p \dx< \infty$.
\end{proof}
\subsection{Weighted Lebesgue space embedded in $\mathcal{H}_p(\Om,V)$} In this subsection, we show that there are certain weighted Lebesgue spaces that are continuously embedded in $\mathcal{H}_p(\Om,V)$.
\begin{theorem} \label{Thm:Weighted}
Let $p\geq 2$ and ${\bf{u}} \in W^{1,p}_{\loc}(\Om)\cap C(\Om)$ be a positive solution of $Q_{p,I,V}'[u]=0$ in $\Gw$. For some $\al \in [1,N/p],\beta \in [1,\infty],$ and $r>1$, suppose that there exists $C_r>0$ such that  \begin{align}\label{weighted_sob}
    |A|^{\frac{1}{N}+\frac{1}{\al' p} - \frac{1}{p}} \left[ \frac{1}{|A|} \int_A {\bf{u}}^{{\frac{\al' pr}{\beta'}}} \dx\right]^{\frac{1}{\al' pr}} \left[ \frac{1}{|A|} \int_A {\bf{u}}^{-p'r} \dx\right]^{\frac{1}{ p'r}} & < C_r,   
\end{align}
for any measurable set $A$ in $\Om$. Then, $L^{\al}(\Om,{\bf{u}}^{\frac{\al}{\beta}p}) \subset \mathcal{H}_p(\Om,V)$.
\end{theorem}
\begin{proof}  By \eqref{weighted_sob} and \cite[Theorem 1]{Sawyer}, the following weighted Sobolev inequality holds 
\begin{align} 
  \left[\int_{\Om}  {\bf{u}}^{{\frac{\al' p}{\beta'}}} |\phi|^{\al'p} \dx\right]^{\frac{1}{\al' p}} \leq C \left[\int_{\Om}  {\bf{u}}^{ p} |\nabla \phi|^{p} \dx\right]^{\frac{1}{p}} \qquad \forall\phi \in W^{1,p}(\Om) \cap C_c(\Om)
  \label{Swayer1} 
\end{align}
for some $C>0.$
 Now, for any $\phi \in W^{1,p}(\Om) \cap C_c(\Om)$,
using the simplified energy (Proposition \ref{Prop:simp_energy}) and the assumption $p\geq 2$, we obtain
 \begin{align*}
     \int_{\Om} |g| |{\bf{u} }\phi|^p \dx&= \int_{\Om} |g| {\bf{u}}^{\frac{1}{\beta} p + (1-\frac{1}{\beta})p} |\phi|^p\dx  \leq \left[\int_{\Om} |g|^{\al} {\bf{u}}^{ \frac{\al}{\beta}p} \dx\right]^{\frac{1}{\al}} \left[\int_{\Om} {\bf{u}}^{{\frac{\al' p}{\beta'}}} |\phi|^{\al'p} \dx\right]^{\frac{1}{\al'}} \\
     & \leq \left[\int_{\Om} |g|^{\al} {\bf{u}}^{ \frac{\al}{\beta}p} \dx\right]^{\frac{1}{\al}} \left[\int_{\Om}  {\bf{u}}^{ p} |\nabla \phi|^{p}\dx \right] 
     \leq C_1 \left[\int_{\R^N} |g|^{\al} {\bf{u}}^{ \frac{\al}{\beta}p} \dx \right]^{\frac{1}{\al}} Q_{p,I,V}[{\bf{u}} \phi] \,,
 \end{align*}
 for some $C_1>0$. This proves our claim.
\end{proof}
\begin{example}{\em 
$(i)$ Let $2 \leq p <N, V = 0$. Then ${\bf{u}} = 1$ is a positive solution of $-\Gd_p[v]=0$ in $\R^N$. By taking $\al=\beta ={N}/{p}$, the  Gagliardo-Nirenberg inequality implies \eqref{Swayer1}. Hence, $L^{{N}/{p}}(\R^N)\subseteq \mathcal{H}_p(\R^N,{0})$. Consequently, $L^{{N}/{p}}(\R^N)\subset \mathcal{H}_p(\R^N,V)$ for  $2 \leq p <N$, and $V \geq 0$.

$(ii)$ Let $2 \leq p<N$, $V(x)= -[\frac{N-p}{p}]^p |x|^{-p}$, and $\Om=B_1(0)$. Then ${\bf{u}}= |x|^{\frac{p-N}{p}}$ is a positive solution of $Q_{I,p,V}'[u]=0$. For $\al = N/p$, and $\beta =1$, \eqref{Swayer1} is satisfied for any $r>1$. Hence, $L^{N}(B_1(0),|x|^{N(p-N)}) \subset \mathcal{H}_p(B_1(0),V)$.}
\end{example} 
\begin{proposition} \label{Prop:Lebesgue}
Let $V = 0$ and $A \in L^{\infty}(\Om;\R^{N\times N})$. Then for  $p \in (1,N)$ and an open subset $\Om$  in  $\R^N$, $L^{{{N}/{p}}}(\Om)$ is continuously embedded in $\H_p(\Om,{\bf{0}})$.
 \end{proposition}
 \begin{proof}
 For $g \in L^{{{N}/{p}}}(\Om)$, using Remark~\ref{Cap_def_eqiv} ({\em v}), and recalling that $p^*=pN/(N-p)$, we have
 \begin{align*}
 \frac{\int_{F}|g|\dx}{{\rm{Cap}}_{\bf{1}}(F,\Om)} \leq \frac{\left[\int_{F}|g|^{\frac{N}{p}}\dx\right]^{\frac{p}{N}}|F|^{\frac{N-p}{N}}}{{\rm{Cap}}_{\bf{1}}(F,\Om)} \leq C\|g\|_{L^{{N}/{p}}(\Om)}
 \end{align*}
 for any compact $F \Subset \Om$. Thus, $\|g\|_{\H_p(\Om,{\bf{0}})} \leq C \|g\|_{L^{{N}/{p}}(\Om)}$. This proves the proposition.
 \end{proof}
\section{Compactness of the functional $T_g$}
Suppose that $Q_{p,A,V}$ is subcritical with $V \geq 0$, and  $g \in \mathcal{H}_p(\Om,V)$. In other words, $g$ satisfies
\begin{align} \label{Hardy:0}
  \int_{\Om} |g||\phi|^p \dx \leq C \int_{\Om} [|\nabla \phi|_A^p+V|\phi|^p] \dx \qquad \forall \phi \in W^{1,p}(\Om) \cap C_c(\Om) \,, 
\end{align}
for some $C>0$.  Then the best constant in \eqref{Hardy:0} is achieved in $\mathcal{D}_{A,V}^{1,p}(\Om)$ if the functional
$$T_g(\phi)=\int_{\Om} |g||\phi|^p \dx $$
is compact in $\mathcal{D}_{A,V}^{1,p}(\Om)$ (for instance see \cite[Section 6.2]{Anoopthesis}).  We devote this section to discuss the compactness of the map $T_g$.
Let us commence with the following important remarks.
\begin{remarks} \label{Remark:cpct} {\em
$(i)$ Let $\Gw\subset \R^N$ be a domain. Recall that for a nonnegative $g \in L^{1}_{\loc}(\Om)$,  the space $\mathcal{D}_{A,g}^{1,p}(\Om)$ is the  completion of $W^{1,p}(\Om)\cap C_c(\Om)$ with respect to the norm 
$$\norm{\phi}_{{\mathcal{D}_{A,g}^{1,p}}} := \left[\||\nabla \phi|_A\|_{L^p(\Om)}^p + \|\phi\|_{L^p(\Om,g \dx)}^p \right]^{\frac{1}{p}}.$$ 
It is known that $\mathcal{D}_{A}^{1,p}(\Om):=\mathcal{D}_{A,0}^{1,p}(\Om)$ is not a Banach space of functions if and only if  $Q_{p,A,{\bf{0}}}$ is critical. However, if $Q_{p,A,{\bf{0}}}$ is subcritical in $\Gw$, then for $ g\in \mathcal{H}_p(\Om,{\bf{0}}) \setminus \{0\}$, the norm $\norm{\phi}_{{\mathcal{D}_{A,|g|}^{1,p}}}$ is equivalent to $\norm{\phi}_{{\mathcal{D}_{A,0}^{1,p}}}$ for $\phi \in W^{1,p}(\Om)\cap C_c(\Om)$, and $\mathcal{D}_{A,|g|}^{1,p}(\Om)=\mathcal{D}_{A}^{1,p}(\Om)$.
  Thus, the Hardy-type inequality \eqref{Hardy:0} with Hardy-weight $g$ holds for all $\phi \in \mathcal{D}_A^{1,p}(\Om)$ as well. 

$(ii)$  In view of $(i)$, if $ g\in \mathcal{H}_p(\Om,{\bf{0}}) \setminus \{0\}$, then it follows that $\mathcal{D}_{A}^{1,p}(\Om) \hookrightarrow L^p(\Om,|g|\dx)$, and hence, $\mathcal{D}_{A}^{1,p}(\Om)$ is a well defined Banach space of functions. Furthermore,   $\mathcal{D}_{A}^{1,p}(\Om)$ is reflexive and separable. Indeed, consider the space $L^p_{A}(\Om)^{N}$ with the norm $\|f\|_{L^p_A(\Om)}:= \left[\int_{\Om} |f|_A^p \dx \right]^{{1}/{p}}$. Consider the map 
	$$E: \DpA \to L^p_{A}(\Om)^{N} \quad \mbox{ given by } \quad E(\phi):=\left(\frac{\partial \phi}{\partial x_1}, \ldots,\frac{\partial \phi}{\partial x_N}\right).$$ 
	Clearly, $E$ is an isometry from $\DpA$ onto a closed subspace of $L^p_{A}(\Om)^{N}$. Thus, it inherits the reflexivity and separability properties of $L^p_{A}(\Om)^{N}$. Therefore, it is enough to prove that $(L^p_{A}(\Om)^{N},\|f\|_{L^p_A(\Om)})$ is a reflexive Banach space. Recall that the local uniform ellipticity of $A$ ensures that there exists a positive function $\theta: \Om \to \R$ such that
\begin{align*} 
 \frac{1}{\theta(x)} |\xi|_{I} \leq |\xi|_{A} \leq \theta(x) |\xi|_{I} \,, \ \forall x \in \Om, \xi \in \R^N \,.   
\end{align*}  
Consequently, 
$$\|f\|_{L^p_I(\Om,\theta^{-1}\dx)} \leq \|f\|_{L^p_A(\Om)} \leq \|f\|_{L^p_I(\Om,\theta \dx)}.$$
Using standard arguments, one obtains that $(L^p_{I}(\Om)^{N},\|f\|_{L^p_I(\Om,\theta^{k}\dx)})$ is a reflexive Banach space for $ k=1,-1$ and $1<p<\infty$. Hence, $(L^p_{A}(\Om)^{N},\|f\|_{L^p_A(\Om)})$ is a reflexive and separable Banach space.  

$(iii)$ Let $(u_n)$ be a sequence in $W^{1,p}(\Om)\cap C_c(\Om)$ such that $u_n \wra u$  in $\mathcal{D}_{A}^{1,p}(\Om)$. Consider any open smooth bounded set $\widetilde{\Om}$ in $\Om$. It is clear that $u_n \in W^{1,p}(\widetilde{\Om})$ and  $(u_n)$ is bounded in $W^{1,p}(\widetilde{\Om})$ (by Sobolev inequality and the local uniform ellipticity of $A$). Thus, (up to a subsequence) $u_n|_{\widetilde{\Om}} \wra u|_{\widetilde{\Om}}$ in $W^{1,p}(\widetilde{\Om})$. Hence, (up to a subsequence) $u_n|_{\widetilde{\Om}} \ra u|_{\widetilde{\Om}}$ in $L^p(\widetilde{\Om})$ (by Rellich-Kondrachov compactness theorem). }
\end{remarks} 
\begin{remark} \label{Remark_ref} \rm
 Let $V\in \mathcal{M}^{q}_{\text{\loc}}(p;\Om)$, and assume that $Q_{p,A,V}$ is subcritical in $\Gw$. Then,
\begin{align} \label{WHV_}
\int_{\Om} V^-|\phi|^p \dx \leq  \int_{\Om} [|\nabla \phi|_A^p + V^+|\phi|^p]\dx \qquad \forall \phi \in W^{1,p}(\Om) \cap C_c(\Om),
\end{align}
Using a standard completion argument as in Remark \ref{Remark:cpct}-$(i)$, we infer that \eqref{WHV_} holds for all $\phi \in \DpAV$ as well.
Now, it is easy to see that, if $g \in \mathcal{H}_p(\Om,V)$, then
\begin{align} \label{WHV3}
\int_{\Om} |g||\phi|^p \dx \leq \S_g(\Gw) \int_{\Om} [|\nabla \phi|_A^p+V|\phi|^p]\dx \qquad \forall \phi \in \DpAV.
\end{align}
Moreover, using a similar argument as in Remark \ref{Remark:cpct}-$(ii)$, it follows that $\DpAV$ is a reflexive, separable Banach space.
\end{remark}

For $A=I_{N \times N}$ and $V=0$, there are necessary and sufficient conditions on $g \in \mathcal{H}_p(\Om,{\bf{0}})$ for the compactness of $T_g$ on $\mathcal{D}^{1,p}_I(\Om)$, for instance, see \cite[Section 2.4.2]{Mazya_book}, and \cite[Theorem 8]{AU}. In Theorem \ref{Cpct_Char} we prove an analogous characterization for a general matrix $A$ and a nonnegative potential $V$. However, the main aim of this section is to provide a subspace of $\mathcal{H}_p(\Om,V)$ for which $T_g$ is compact on $\DpAV$. To this end, we define
 $$\mathcal{H}_{p,0}(\Om,V):=\overline{\mathcal{H}_{p}(\Om,V) \cap L^\infty_c(\Gw)}^{\|\cdot\|_{\mathcal{H}_{p}(\Om,V)}}  .$$ 
\begin{theorem} \label{Thm:compactness}
	Let $g \in \mathcal{H}_{p,0}(\Om,V^+)$. Then $T_g$ is compact in $\mathcal{D}_{A,V^+}^{1,p}(\Om)$.
\end{theorem}
\begin{proof}
Let $g \in \mathcal{H}_{p,0}(\Om,V^+)$, and $(\phi_n)$ be a bounded sequence in $\DpAV$. By reflexivity, up to a subsequence, $\phi_n \wra \phi$ in $\DpAV$. Choose $ \vge>0$ arbitrary. Then there exists $g_{\vge } \in \mathcal{H}_{p}(\Om,V^+) \cap L^\infty_c(\Gw)$ such that  $\|g - g_{\vge }\|_{\H_p(\Om,V^+)} < \vge$. Let $K_{\vge }$ be the support of $g_{\vge }$ and $K_{\vge } \Subset \Om_{\vge } \Subset \Gw$, where $\Om_{\vge }$ is an smooth open, bounded set. Notice that \begin{align} \label{Req1:cpct}
|T_g(\phi_n) - T_g(\phi)| & \leq \int_{\Om} |g - g_{\vge }| ||\phi_n|^p-|\phi|^p| \dx + \left|\int_{\Om} | g_{\vge }| [|\phi_n|^p-|\phi|^p] \dx \right| \nonumber \\
&:= I_1 +I_2 \,.
\end{align}
Using the definition of
$\|\cdot\|_{\H_p(\Om,V^+)}$, the first integral can be estimated as follows:
\begin{align} \label{Req:cpct}
I_1 & \leq C \|g - g_{\vge }\|_{\H_p(\Om,V^+)} \left[\|\phi_n\|^p_{\mathcal{D}_{A,V^+}^{1,p}(\Om)} +\|\phi\|^p_{\mathcal{D}_{A,V^+}^{1,p}(\Om)} \right] \leq C M \vge   \,,
\end{align}
for some $M>0$ independent of $n$ (as $(\phi_n)$ is bounded in $\DpAV$.) 

In order to estimate the second integral, we notice that $ \phi_n|_{\Om_{\vge }} \in W^{1,p}(\Om_{\vge })$ for each $n \in \N$, and $( \phi_n)$ is bounded in $W^{1,p}(\Om_{\vge })$ (by local uniform ellipticity). Since, $\Om_{\vge }$ is a smooth bounded domain we obtain by compactness that $\lim_{n\to\infty}\|\phi_n-\phi\|_{L^p(\Om_{\vge })}=0$ and $\phi_n \ra \phi$ a.e. (up to a subsequence). Consequently, $\lim_{n \ra \infty} \int_{\Om_{\varepsilon}} |\phi_n|^p \dx = \int_{\Om_{\varepsilon}} |\phi|^p \dx$. Now, for each $\varepsilon>0$, we have $|g_{\varepsilon}| |\phi_n|^p \leq \|g_{\varepsilon}\|_{L^{\infty}(\Om)} |\phi_n|^p$ for all $n \in \N$, and $|g_{\varepsilon}| |\phi_n|^p \ra |g_{\varepsilon}| |\phi|^p$ a.e.. Thus, by the generalized dominated convergence theorem \cite[Section 4, Theorem 16]{Royden}, we infer that $\lim_{n \ra \infty} \int_{\Om_{\varepsilon}} |g_{\varepsilon}||\phi_n|^p \dx = \int_{\Om_{\varepsilon}} |g_{\varepsilon}||\phi|^p \dx$ for each $\varepsilon>0$. Thus, by taking $n \ra \infty$ in \eqref{Req1:cpct} and using \eqref{Req:cpct}, we obtain
$$\limsup_{n \ra \infty} |T_g(\phi_n)-T_g(\phi)| \leq CM\vge \,.$$
Since  $CM$ is independent of $\vge$, and $\vge >0$ is arbitrary, we are done.
\end{proof} 
\begin{remarks} \label{Rmk:Hptilde} \rm
$(i)$ Clearly,  $\mathcal{H}_{p}(\Om,V) \subset \mathcal{H}_{p}(\Om,V^+)$, and $\|\cdot\|_{\mathcal{B}(\Om,V^+)} \leq \|\cdot\|_{\mathcal{B}(\Om,V)}$. Thus, $\mathcal{H}_{p}(\Om,V) \cap L^{\infty}_c(\Om) \subset \mathcal{H}_{p}(\Om,V^+) \cap L^{\infty}_c(\Om).$ Hence, 
$$\overline{\mathcal{H}_{p}(\Om,V) \cap L^{\infty}_c(\Om)}^{\|.\|_{\mathcal{B}(\Om,V)}} \subset \overline{\mathcal{H}_{p}(\Om,V) \cap L^{\infty}_c(\Om)}^{\|.\|_{\mathcal{B}(\Om,V^+)}} \subset   \overline{\mathcal{H}_{p}(\Om,V^+) \cap L^{\infty}_c(\Om)}^{\|.\|_{\mathcal{B}(\Om,V^+)}}.$$ 
Therefore,  $\mathcal{H}_{p,0}(\Om,V) \subset \mathcal{H}_{p,0}(\Om,V^+)$.	 It follows from Theorem~\ref{Thm:compactness} that if  $g \in \mathcal{H}_{p,0}(\Om,V)$, then $T_g$ is compact in $\mathcal{D}_{A,V^+}^{1,p}(\Om)$.

$(ii)$ In fact, if $p \in (1,N)$, $A=I_{N\times N}$ and $V=0$, then the converse of Theorem \ref{Thm:compactness} is also true \cite[Theorem 8]{AU}. The embedding $\mathcal{D}_I^{1,p}(\Om) \hookrightarrow L^{p^*}(\Om)$ played a crucial role in their proof. For $p \in (1,N)$, $V=0,$ and a matrix $A$ with local ellipticity function $\theta \in L^{\infty}(\Om)$, it clearly follows  that $\DpA=\mathcal{D}^{1,p}_I(\Om)$ as vector spaces having equivalent norms.  Indeed, let $\gth \in L^\infty(\Gw)$. Then
	\begin{align*}
	\int_\Gw |\nabla\phi|^p\dx&=\int_\Gw \gth^p\frac{1}{\gth^p}|\nabla\phi|^p\dx\leq 
	\int_\Gw \gth^p|\nabla\phi|_A^p\dx \leq \|\gth\|_{L^\infty(\Gw)}^p\int_\Gw |\nabla\phi|_A^p\dx\\
	&\leq \|\gth\|_{L^\infty(\Gw)}^p\int_\Gw \gth^p|\nabla\phi|^p\dx \leq
	\|\gth\|_{L^\infty(\Gw)}^{2p}\int_\Gw |\nabla\phi|^p\dx. 
	\end{align*}
	 Also, the $Q_{p,A,{\bf{0}}}$-capacity of compact sets is equivalent to the corresponding $Q_{p,I,{\bf{0}}}$-capacity. Therefore, if $\theta \in L^{\infty}(\Om)$, $V=0$, and $T_g$ is compact in $\mathcal{D}_A^{1,p}(\Om)$, then $g \in \mathcal{H}_{p,0}(\Om,0)$.

$(iii)$ Let $1<p<N$, $V=0,$ and $\theta \in L^{\infty}(\Om).$ Let $g \in \mathcal{H}_{p}(\Om,0) \cap L^\infty_c(\Gw)$  with compact support $K_g \Subset \Gw$. Then, using Remark \ref{Cap_def_eqiv}-$(v)$, we obtain
\begin{align*}
 \|g\|_{\H_p(\Om,{\bf{0}})} & =   \sup\left\{ \frac{\int_{F} |g| \dx}{{\rm{Cap}}_{\bf{1}}(F,\Om)} \, \, \biggm| \,  F \Subset \Om; {\rm{Cap}}_{\bf{1}}(F,\Om)\ne 0 \right\} \\
 & \leq    \sup\left\{\frac{\int_{F} |g| \dx}{{\rm{Cap}}_{\bf{1}}(F,\Om)} \, \, \biggm| \,  F \ \text{compact in} \ K_g; {\rm{Cap}}_{\bf{1}}(F,\Om)\ne 0 \right\} \\
 & \leq    \sup \! \left\{ \frac{\left[\int_{F} |g|^{\frac{N}{p}} \dx \right]^{\frac{p}{N}}\! |F|^{\frac{p}{p^*}}}{{\rm{Cap}}_{\bf{1}}(F,\Om)}  \, \biggm|   F  \text{ compact in }  K_g; \  {\rm{Cap}}_{\bf{1}}(F,\Om)\ne 0 \!\right\} 
  \leq C \|g\|_{L^{\frac{N}{p}}(\Om)} ,
\end{align*} 
for some $C>0$. Consequently,
$L^{\frac{N}{p}}(\Om) \hookrightarrow \H_{p,0}(\Om,{\bf{0}})$. 
\end{remarks}
\section{Attainment of best constant I}
In the present section we assume  that for $p\leq N$,  the following additional regularity assumption holds in a certain subdomain $\Gw'\Subset \Gw$:
$${\bf{(H0)}} : A\!\in \! C^{0,\gamma}_{\rm loc}(\Gw';\R^{N\times N}), \mbox{ and } g, \!V \!\in \! \M^q_{\loc}(p;\Gw'),\mbox{ where }  0\!< \!\gg \! \leq  \! 1, q\! > \! N, \mbox{ and } g \! \in \!  \mathcal{H}_p(\Om,V).$$
Under hypothesis (H0) positive solutions $v$ of the equation $Q'_{p,A,V-\gl g}[u]=0$ in $\Gw'$  \cite[Theorem~5.3]{Lieberman} are differentiable in $\Gw'$ and satisfy in any $\gw\Subset \Gw'$
$$\sup_{\om}|\nabla v|\leq C  \sup_{\Gw'}|v|,$$
for some positive constant $C$, depending only on $n,p,\gamma,q,{\mathrm{dist}}(\om,\Gw')$,   $\|A\|_{C^{0,\gamma}(\Gw')}$,  $\|\gth^{-1}\|_{L^\infty(\Gw')}$, $\|g\|_{M^q(\Gw')}$, and $\|V\|_{M^q(\Gw')}$.

 Recall that in view of \eqref{WHV3} the best constant  $\mathcal{B}_g(\Om)$  in \eqref{WH} for  $g \in  \mathcal{H}_p(\Om,V)$ is given by
$$\frac{1}{\mathcal{B}_g(\Om)}= \S_g(\Om)=\inf \{Q_{p,A,V}(\phi) \bigm| \phi \in \DpAV \,, \int_{\Om} |g| |\phi|^p \dx =1\} .$$

In the present and the following section we prove certain sufficient conditions on $g$ and $V$ so that the best constant in \eqref{WH} is attained in $\DpAV$. Let $\bigcup_n \Om_n= \Om$ be a compact smooth exhaustion of $\Om$. We define the following: 
\begin{align*}
    \S_{g}^{\overline{\infty}}(\Gw) & :=  \lim_{i\ra \infty} \inf \left \{Q_{p,A,V}(\phi) \bigm|  \phi \in \mathcal{D}^{1,p}_{A,V^+} (\Om \setminus \overline{\Gw_i}), \ \int_{\Om \setminus \overline{\Gw_i}} |g| |\phi|^p  \dx =1 \right \} .
\end{align*}
 Clearly, $\S_g(\Om) \leq \S_g^{\overline{\infty}}(\Gw)$. Using criticality theory we prove below that under some regularity assumptions if the spectral gap condition $\S_g(\Om) < \S_g^{\overline{\infty}}(\Gw)$ holds, then $\mathcal{B}_g(\Om)$ is attained.
\begin{theorem} \label{Thm:best_constant}
Let $g \!\in \!\mathcal{H}_p(\Om,V)\cap \mathcal{M}^{q}_{\loc}(p;\Om)$ be such that $\S_g(\Om) < \S_g^{\overline{\infty}}(\Gw)$.  Let $K\Subset K_1\Subset \Gw'$, where  $K$ is a compact set such that $Q_{p,A,V-\gl_1 g} \geq 0$ in $\Gw\setminus K$ for some 
$\gl_1 \in (\S_g(\Om),\S_g^{\overline{\infty}}(\Gw))$, and for $p\leq N$, $\Gw'\subset \Gw$ is a subdomain in which hypothesis \rm{(H0)} is satisfied. Let $G$ be a positive solution the equation $Q'_{p,A,V-\S_g(\Om)g}[u]=0$ in $\Gw\setminus K$ of minimal growth in a neighborhood of infinity in $\Gw$. 
Assume that $\int_{\Gw\setminus K_1} V^{-} G^p\dx <\infty$. Then, $\mathcal{B}_g(\Om)$ is attained in $\DpAV$. 
\end{theorem}
We begin  with a key lemma, claiming that if the  above spectral gap condition  is satisfied for a Hardy-weight $g$, then $Q_{p,A,V-\S_g(\Om)|g|}$ is critical.  The proof is similar to the proof of \cite[Lemma~2.3]{Lamberti}.
\begin{lemma}\label{lem-spectral-gap}
	 Let $g \in \mathcal{H}_p(\Om,V)\cap \mathcal{M}^{q}_{\loc}(p;\Om)$ be such that $\S_g(\Om) < \S_g^{\overline{\infty}}(\Gw)$. For $p\leq N$, assume further that the hypothesis {\rm{(H0)}}  is satisfied in a subdomain $\Gw'$, and that there exists a smooth open  set $K_0\Subset  \Gw'$  such that $Q_{p,A,V-\gl_1|g|} \! \geq \!0$ in $\Gw\setminus \bar{K_0}$ for some $\gl_1 \in (\S_g(\Om),\S_g^{\overline{\infty}}(\Gw))$.  Then $Q_{p,A,V-\S_g(\Om)|g|}$ is critical in $\Gw$. 
\end{lemma}
\begin{proof}
		Following   \cite[Lemma~4.6]{pin90} and \cite[Lemma~2.3]{Lamberti}, we set
	$$S\!:=\!\{t\!\in\! {\mathbb{R}}\!\mid \! Q_{p,A,V-t|g|}\!\geq\! 0 \mbox{ in }\Gw \},\   S_{\infty}\!:=\!\{  t\!\in\! {\mathbb{R}}\!\mid\!  Q_{p,A,V-t|g|}\!\geq\! 0 \mbox{ in }  \Omega\! \setminus \!\bar K \mbox{ for some } K\!\Subset \Omega  \}.$$
	Clearly, $S$ and $S_\infty$ are intervals, and since $Q_{p,A,V}$ has  a spectral gap, it follows that
	$$S=\ (\!-\infty , \S_g(\Om)] \ \varsubsetneq \ S_{\infty} \subset \ (\!-\infty , \S_g^{\overline{\infty}}(\Gw)].$$
	For simplicity, we set $\lambda_0=\S_g(\Om)$.
	Let $\lambda _1 \in S_{\infty }\setminus S$ and $K_0\Subset \Gw'$ satisfy the hypothesis of the lemma. Consequently,  the equation  $Q'_{p,A,V-\gl_1|g|}[u]=0$ in $\Gw\setminus \bar{K_0}$ admits a positive solution.  
	
	{\bf Claim} There exists $0\lneqq \mathcal{V}\in \mathcal{M}^{q}_{c}(p;\Om)$ such that $Q_{p,A,V-\gl_1|g|+\mathcal{V} }\geq 0$ in $\Gw$.

	Fix a smooth open  set $K$ satisfying  $K_0\Subset K\Subset  \Gw'$.  We first show that there exists a positive solution $v$ of the equation  $Q'_{p,A,V-\gl_1|g|}[u]=0$ in $\Gw\setminus \bar{K}$ satisfying $v=0$ on $\partial K$.
	
	To this end,  consider  a smooth exhaustion $\{\Gw_i\}_{i\in\N}$ of $\Gw$ by smooth relatively compact subdomains such that $x_0\in\Gw_1\setminus \bar{K}$    and such that 
	$\bar{K}\subset \Omega_{i-1}\Subset \Omega_i $ for all $i>1$.  Let $v_{i}$ be the unique positive solution of the Dirichlet problem 
	\[
	\left\{
	\begin{array}{ll}
	Q'_{p,A,V-\gl_1|g|}[u]=f_i & \mbox{in }\Gw_i\setminus \bar{K},
	\\[8pt]
	u=0 & \mbox{on }\partial (\Gw_i\setminus \bar{K}),
	\end{array}
	\right.
	\]
	where $f_i$ is  a nonzero nonnegative function in $C_c^\infty(\Gw_i\setminus \overline{\Gw_{i-1} })$      normalized in such a way that  $v_i(x_0)=1$. The existence and uniqueness of such a solution is guaranteed by \cite[Theorem~3.10]{Yehuda_Georgios} combined with the fact that $Q'_{p,A,V-\gl_1|g|}[u]=0$  admits a positive solution in $\Gw\setminus \bar{K_0}$. 
	
	By the Harnack principle and elliptic regularity (see for example, \cite{Yehuda_Georgios}) the sequence $\{v_i\}_{i\in\N}$ admits a subsequence converging locally uniformly to a positive solution $v$ of the equation  $Q'_{p,A,V-\gl_1|g|}[u]=0$ in $\Gw\setminus \bar{K}$ satisfying $v=0$ on $\partial K$.
	Note that by classical regularity theory we  have that $v$ is of class $C^{\alpha }$ up to $\partial K$. Moreover, for $p\leq N$,
	hypothesis (H0) implies that $v$ has  a locally bounded gradient in $\Gw'\setminus K$.

Let  $K_1$  be an open set such that  $K\Subset K_1\Subset  \Gw'$ and let 
$\min_{x\in \partial K_1} v(x)=m>0$. 
Fix $\varepsilon>0$ satisfying  $8\varepsilon < m$.  Let  $F$ be a $C^{2}$  function from $[0, +\infty [$ to $[0,+\infty [$ such that 
$F(t)=\varepsilon$ for all $0\le t \le 2\varepsilon $ and $F(t)= t$ for all $t\geq  4\varepsilon $  and such that $F'(t)\ne 0$ for all $t>2\varepsilon $. Assume also that 
$|F'(t)|^{p-2}F''(t)\to 0$ as $t\to 2\varepsilon$, hence the function $t\to |F'(t)|^{p-2}F''(t)$  (defined identically equal to zero on $ [0,2\varepsilon]$)  is continuous  on $[0, +\infty [$ (for this purpose, it is enough  for example that $F$ is chosen to be of the type $\varepsilon + (t-2\varepsilon)^{\beta }$ for all $t>2\varepsilon$ sufficiently close to $2\varepsilon$ and $\beta >\max\{p',2\}$).       

We set $\bar v(x) := F( v(x))$ for all   $x\in K_1\cap (\Omega \setminus {\bar  K})$.  By the definition of $\bar v$, there exists an open  neighborhood $U$ of $\partial K$ such that $\bar v (x) =\varepsilon $ for all $x\in U \cap (\Omega \setminus {\bar  K})$, and there exists an open neighborhood  $U_1$ of $\partial K_1$ such that $\bar v (x)=v(x)$ for all $x\in U_1\cap K_1$ . 
Thus $\bar v (x) $ can be extended continuously into the whole of $\Omega $ by setting $\bar v (x)=\varepsilon $ for all $x\in \bar K$  and $\bar v(x)= v(x)$ for all $x\in \Omega \setminus K_1$. 
By   Lemma~\ref{weak_lapl} in Appendix~\ref{appendix1}, we have in the weak sense that 
\begin{equation}
\label{superposition} 
-\Delta_{p,A} \bar v =-|F'(v)|^{p-2}[(p-1)F''(v)|\nabla v|_A^p+F'(v)\Delta _{p,A}v ] \qquad \mbox{in } K_1 \setminus \bar K. 
\end{equation}
By our assumptions on $F$ and $v$, it follows that $-\Delta_{p,A} \bar v \in L^1_\loc (\Gw)$, and
$$-\Delta_{p,A} \bar v = 
\begin{cases}
0 &  \mbox{in } K \! \cup \!  (U \!\cap\! (\Omega \!\setminus \!{\bar K})),\\
(\lambda_1 g-V) v^{p-1}&   \mbox{in } (\Omega\!\setminus \!\bar K_1) \!\cup\! (U_1\! \cap \!  K_1),\\
-|F'( v)|^{p-2}[(p \! - \! 1)F''(v)|\nabla v|_A^p \! + \!F'(v)\Delta _{p,A}v] &
\quad \mbox{ otherwise}. 
\end{cases}
$$

 Define the potential ${\mathcal{V}}$ by setting
	$$
	{\mathcal  {V}}= \frac{|Q'_{p,A,V-\gl_1|g|}[\bar v]   | }{\bar v ^{p-1}} \,.
	$$
	Clearly, ${\mathcal  {V}} \in  L^1_c(\Gw)$ which for $p> N$ means that ${\mathcal  {V}} \in  \mathcal{M}^{q}_{c}(p;\Om)$. Moreover, for $p \leq N$, Hypothesis (H0) implies that ${\mathcal  {V}} \!\in \! \mathcal{M}^{q}_{c}(p;\Om)$. Therefore,  
	$\bar v$ is a weak positive supersolution of  $Q'_{p,A,V-\gl_1|g|+\mathcal{V}}[u] = 0$ in $\Gw$. Hence,   $Q_{p,A,V-\gl_1|g|+\mathcal{V} }\geq 0$ in $\Gw$,  and the Claim is proved.
	
	We set $\lambda_t=t\lambda_1+(1-t)\lambda_0$. By using \cite{Yehuda_Georgios}, it follows that the set
	$$\{(t,s)\in [0,1]\times {\mathbb{R}} \mid  Q_{p,A,V-\gl_t|g|+s\mathcal{V}}\geq 0 \mbox{ in } \Gw\}$$  is a convex set. Hence, the function $\nu:[0,1]\to \R$  defined by
	$$\nu (t):=\min \{s\in {\mathbb{R}} \mid Q_{p,A,V- \gl_t|g|+s\mathcal{V}}\geq 0 \mbox{ in } \Gw \}$$ is convex . Since $\mathcal{V}$ has compact support,
	\cite[Proposition~4.19]{Yehuda_Georgios} implies that $Q_{p,A,-\lambda_t|g| +\nu (t)\mathcal{V} }$ is critical for all $t\in [0,1]$. We note that by definition $\nu (t)>0$ for all $t\in (0,1]$, while
	$\nu (0)\le 0$. Since $\nu $ is convex, we have $\nu (0)=0$, and hence, $Q_{p,A,V-\S_g(\Om)|g|}$ is critical in $\Gw$. 
\end{proof}
\begin{definition}[Null sequence and ground state] \label{def-gs}{\em 
A nonnegative sequence $(\phi_n) \in W^{1,p}(\Om) \cap C_c(\Om)$  is called a null-sequence with respect to the nonnegative functional $Q_{p,A,V}$ if
\begin{itemize}
    \item there exists a subdomain $O \Subset \Om$ such that $\|\phi_n\|_{L^p(O)}  \asymp 1$ for all $n \in \N,$
    \item $\lim_{n \ra \infty} Q_{p,A,V}(\phi_n)=0$. 
\end{itemize}
We call a positive function $\Phi \in W^{1,p}_{\loc}(\Om) \cap C(\Om)$ a ground state of $Q_{p,A,V}$ if $\phi$ is an $L^p_{\loc}(\Om)$ limit of a null-sequence.}
\end{definition}
\begin{remark}\label{rem-gs}{\em 
	A nonnegative functional $Q_{p,A,V}$ is critical in $\Gw$ if and only if $Q_{p,A,V}$ admits a null-sequence in $\Gw$. Moreover, any null-sequence converges weakly in $L^{p}_\loc(\Gw)$ to the unique (up to a multiplicative constant) positive (super)solution of the equation $Q'_{p,A,V}[u]=0$ in $\Gw$.  Furthermore,   there exists a null-sequence which converges locally uniformly in $\Gw$ to the ground state, and the ground state is a positive solution of the equation $Q'_{p,A,V}[u]=0$ in $\Gw$ which has minimal growth in  a neighborhood of infinity in $\Gw$  \cite{Arxiv_Pinchover,Yehuda_Georgios}.     }
\end{remark}
\begin{lemma}\label{lem-smaller_null-seq}
Let $(\phi_n) \in W^{1,p}(\Om) \cap C_c(\Om)$ be a null-sequence with respect to the nonnegative functional $Q_{p,A,V}$, and let $\Phi \in W^{1,p}_{\loc}(\Om) \cap C(\Om)$ be the corresponding ground state. For each $n \in \N$, let $\hat{\phi}_n=\min\{\phi_n, \Phi\}$.  Then $(\hat\phi_n)$ is a null-sequence.
\end{lemma}
\begin{proof}
Clearly, $\hat{\phi}_n \in W^{1,p}(\Om) \cap C_c(\Om)$ and $Q_{p,A,V}[\hat{\phi}_n] \geq 0$. We claim $Q_{p,A,V}(\hat{\phi}_n) \leq Q_{p,A,V}(\phi_n)$. Consider the open set 
$O_n:= \{x\in \Gw \mid \phi_n(x)>\Phi(x)\}$. Since $Q_{p,A,V}$ is nonnegative and $\phi$ is a ground state, it follows that $Q_{p,A,V}'[\Phi] = 0$ in $\Gw$. By testing this equation with the nonnegative function ${(\phi_n^p-\Phi^p)_+}/{\Phi^{p-1}}$, we obtain
\begin{align} \label{ineq:Bianchi}
     &0 \!\leq \!\int_{O_n} |\nabla \Phi|_A^{p-2} \nabla \Phi \!\cdot\!\nabla \left(\frac{\phi_n^p}{\Phi^{p-1}}\right)\dx + \int_{O_n} V\phi_n^p\dx - Q_{p,A,V}(\Phi|_{O_n}) \nonumber \\ 
    & \!= \! p \int_{O_n}\!\!\! \left(\frac{\phi_n}{\Phi}\right)^{p-1}\!\! |\nabla \Phi|_A^{p-2} \nabla \Phi  \!\cdot\!\nabla \phi_n\dx \!- \!(p-1 ) \!\!\int_{O_n} \!\!\!\left(\frac{\phi_n}{\Phi}\right)^{p} |\nabla \Phi|_A^p dx +\!\!\!\int_{O_n} \!\!\!V\phi_n^p \dx - Q_{p,A,V}(\Phi|_{O_n})
\end{align}
Now, by the Picone-type identity \cite[Lemma 4.9]{Arxiv_Pinchover} for $Q_{p,A,V}$, we have 
\begin{equation}\label{eq-picone}
L(\phi_n,\Phi):= |\nabla \phi_n|_A^p +(p-1) \frac{\phi_n^p}{\Phi^p} |\nabla \Phi|_A^p - p \left(\frac{\phi_n}{\Phi}\right)^{p-1} |\nabla \Phi|_A^{p-2} \nabla \Phi  \!\cdot\! \nabla \phi_n \geq 0 .
\end{equation}
Using the above Picone identity in \eqref{ineq:Bianchi} to get
\begin{align*} 
    0 & \leq  \int_{O_n} |\nabla \phi_n|^p \dx- \int_{O_n} L(\phi_n,\Phi)\dx   + \int_{O_n} V\phi_n^p \dx- Q_{p,A,V}(\Phi|_{O_n}) \\
    & \leq \int_{O_n} |\nabla \phi_n|^p \dx  + \int_{O_n} V\phi_n^p \dx- Q_{p,A,V}(\Phi|_{O_n}) = Q_{p,A,V}(\phi_n|_{O_n}) -Q_{p,A,V}(\Phi|_{O_n}) \,.
\end{align*}
Thus, $Q_{p,A,V}(\hat{\phi}_n|_{O_n})=Q_{p,A,V}(\Phi|_{O_n}) \leq Q_{p,A,V}(\phi_n|_{O_n})$. On the other hand, since $\hat{\phi}_n = \phi_n$ on $\Om \setminus O_n$, our claim follows. 

Now, since $\lim_{n \ra \infty} Q_{p,A,V}(\phi_n) =0$, it follows that $\lim_{n \ra \infty} Q_{p,A,V}(\hat{\phi}_n) =0$. By homogeneity of  $Q_{p,A,V}$, we may assume that  $\|\phi_n\|_{L^{p}(O)} =1$ where $O\Subset \Gw$ is the given open set. The dominated convergence theorem implies that $\lim_{n \ra \infty} \|\hat\phi_n\|_{L^{p}(O)} = 1$, and $\hat{\phi}_n \ra \Phi$ in $L^p_{\loc}(\Om)$. Thus, $(\hat{\phi}_n)$ is a null-sequence.
\end{proof}
\begin{proof}[Proof of Theorem~\ref{Thm:best_constant}]
	Lemma~\ref{lem-spectral-gap} implies that  $Q_{p,A,V-\S_g(\Gw)|g|}$ is critical in $\Gw$, let $\Phi$ be its ground state satisfying $\Phi(x_0)=1$, where $x_0\in\Gw$. By Lemma~\ref{lem-smaller_null-seq}, the functional $Q_{p,A,V-\S_g(\Gw)|g|}$ admits a null-sequence $(\hat\phi_n)$ such that $\hat\phi_n \!\leq\! \Phi$ for all $n\!\geq \!1$, and   $\hat\phi_n\to \Phi$ in $L^p_{\loc}(\Om)$.
We have
\begin{align} \label{Req:1}
\int_{\Gw} |\nabla \hat\phi_n|_A^p \dx + \int_{\Gw} V_+| \hat\phi_n|_A^p \dx =  Q_{p,A,V-\S_g(\Gw)|g|}(\hat\phi_n) +\int_{\Gw} V_{-} \hat\phi_n^p \dx+
\S_g(\Gw)\int_{\Gw} |g|\hat\phi_n^p \dx.
\end{align}
Recall that $\lim_{n \ra \infty} Q_{p,A,V-\S_g(\Gw)|g|}(\hat{\phi}_n) =0$. Moreover,   for $\S_g(\Gw)<s<\S_g^{\overline{\infty}}$ there exists $K\Subset \Gw$ such that $Q_{p,A,V-s|g|} \geq 0$ in $\Gw\setminus K$. Therefore, by the Kova\v{r}\'{\i}k-Pinchover necessary condition (see Proposition~\ref{Thm:KP}),  we obtain that $\int_{\Gw} |g| \Phi^p\dx <\infty$. Our assumption on $V^-$ implies that 
	$\int_{\Gw} V^-\Phi^p\dx <\infty$. Consequently, \eqref{Req:1} implies that the sequence 
	$$\left(\int_{\Gw} \big[|\nabla \hat\phi_n|_A^p+V_+| \hat\phi_n|_A^p \big] \dx \right)$$ 
	is a bounded sequence, i.e., $\hat\phi_n$ is bounded in $\DpAV$. Hence, $\hat\phi_n \wra \phi$ in $\DpAV$ for some $\phi \in \DpAV$. Further, since $\hat\phi_n \ra \Phi$ in $L^p_{\loc}(\Om)$, it follows that $\Phi=\phi \in \DpAV.$ Now, by letting $n \ra \infty $ in \eqref{Req:1} and recalling \eqref{WHV3}, we get
$$ \S_g(\Gw) \int_{\Om}  |g||\Phi|^p\dx \leq \int_{\Om} [|\nabla \Phi|_A^p + V|\Phi|^p]\dx \leq \S_g(\Gw) \int_{\Om}  |g||\Phi|^p\dx.$$
Consequently, the best constant $\S_g(\Gw)$ is attained at $\Phi \in \DpAV$. 
\end{proof}
\begin{lemma}\label{lem-Hp-HP0}
Let $g \in \mathcal{H}_p(\Om,V)\cap \mathcal{M}^q_{\loc}(\Om)$, $g\neq 0$, satisfy $\S_g(\Om) = \S_g^{\overline{\infty}}(\Gw)$. Then $g \not\in \mathcal{H}_{p,0}(\Om,V)$. Moreover, if $Q_{p,A,V}$ is subcritical in $\Gw$, and $p=2$ or $V=0$, then  a Hardy-weight $g$ satisfying $\S_g(\Om) = \S_g^{\overline{\infty}}(\Gw)$ exists and $\mathcal{H}_{p}(\Om,V) \neq \mathcal{H}_{p,0}(\Om,V)$. 
\end{lemma}
\begin{proof}
	Without loss of generality, we may assume that $g \geq 0$ and $\S_g(\Om) =1$. Further, by Theorem~\ref{Thm:Char}, we may take $\mathcal{B}_g(\Om,V)$  as the norm on ${\mathcal{H}_{p}}(\Om,V)$. Suppose that there is a nonnegative  sequence  $\mathcal{H}_{p}(\Om,V)\cap L^{\infty}_c(\Om)$ such that 
	$$\lim_{n\to \infty}\|g_n-g\|_{\mathcal{H}_{p}(\Gw,V)} =0.$$  
	Consider an exhaustion $(\Gw_n)$ of $\Gw$ satisfying  $\supp {g_n}\Subset \Gw_n$. Recall that $\mathcal{B}_g(\Gw_1,V) \leq \mathcal{B}_g(\Gw_2,V)$ if $\Gw_1 \subset \Gw_2$. Therefore, in $\Gw^*_n:=\Gw\setminus \Gw_n$, we have 
$$\|g_n-g\|_{\mathcal{H}_{p}(\Gw,V)} \geq \|g_n-g\|_{\mathcal{H}_{p}(\Gw^*_n,V)}=\|g\|_{\mathcal{H}_{p}(\Gw^*_n,V)}=\mathcal{B}_g(\Gw^*_n,V) =1,$$ 
which contradicts the  assumption on $(g_n)$.

Now, if $p=2$ or $V=0$, then the optimal Hardy-weights $W$ constructed in \cite{DFP,DP,V} satisfy the assumption 
$\S_W(\Om) = \S_W^{\overline{\infty}}(\Gw)$. 
\end{proof}
For  $x\in \overline\Gw$ and $g\in \mathcal{H}_{p}(\Gw,V)$, define   
	$$ \S_{g}(x,\Gw) :=  \lim_{r\ra 0} \inf \left \{Q_{p,A,V}(\phi) \bigm|  \phi \in \mathcal{D}^{1,p}_{A,V^+} (\Om \cap B_r(x)), \ \int_{\Om \cap B_r(x)}\!\! |g||\phi|^p  \dx =1 \right \},$$
and let 
$\Gs_g: = \{x \in \overline\Gw \mid \S_g(x,\Gw) < \infty \}$, and $\S_g^*(\Gw):= \inf_{x \in \overline\Gw} \S_g(x,\Gw)$.
The following lemma claims that if $g \!\in \! \mathcal{M}^{q}_{\loc}(p;\Om)$, then  $\Gs_g \cap \Om \!=\!\phi$.
\begin{lemma}\label{lem-Gsg-0}
	Let $g \in \mathcal{M}^{q}_{\loc}(p;\Om)$, then for any $x\in \Gw$ we have $\S_{g}(x,\Gw)=\infty$. In particular,  $\Gs_g \cap \Om=\phi$. 
\end{lemma}
\begin{proof}
	It is well known \cite[Theorem~13.19]{Leoni} that
	\begin{align} \label{estimate:ev}
	C_0r^{-p}=\inf \left\{\frac{\int_{B_r(0)}|\nabla \phi|^p \dx}{\int_{B_r(0)}|\phi|^p \dx} \, \, \biggm|  \phi \in W^{1,p}(B_r(0)) \cap C_c(B_r(0)) \setminus \{0\} \right\} ,
	\end{align}
	for some $C_0>0.$
	Let $g, V \in \mathcal{M}^{q}_{\text{loc}}(p;\Om)$ and $x \in \Om$. Choose $r>0$ such that $B_r:=B_r(x) \Subset \Om$. Then, by Adams-Morrey inequality (Proposition \ref{Prop:MA}), for any $\vge, \delta >0$ there exists $C(N,p,q),>0$ such that for any $\phi \in W^{1,p}(B_r) \cap C_c(B_r)$
	\begin{align} \label{Req:2}
	\int_{B_r} |V| |\phi|^p \dx \leq \delta \int_{B_r}  |\nabla \phi|^p \dx + \frac{C(N,p,q)}{\delta^{\frac{N}{pq-N}}} \|V\|_{\M^{q}(p;B_r)}^{\frac{pq}{pq-N}} \int_{B_r} |\phi|^p \dx \,,
	\end{align}
	and \begin{align} \label{Req:3}
	\int_{B_r} |g| |\phi|^p \dx \leq \vge \int_{B_r}  |\nabla \phi|^p \dx + \frac{C(N,p,q)}{\vge^{\frac{N}{pq-N}}} \|g\|_{\M^{q}(p;B_r)}^{\frac{pq}{pq-N}} \int_{B_r} |\phi|^p \dx.
	\end{align}
	Now, for any $M>0$, choose $\vge=\frac{\theta_{B_1}}{2M} $, where $\theta_{B_1}>0$ is the local uniform ellipticity constant of $A$ on $B_1$. Then, by \eqref{Req:3}, we have
	\begin{align} \label{Req:4}
	\int_{B_r} |g| |\phi|^p \dx \leq \frac{\theta_{B_1}}{2M} \int_{B_r}  |\nabla \phi|^p \dx + \frac{C(N,p,q)}{[\frac{\theta_{B_1}}{2M}]^{\frac{N}{pq-N}}}  \|g\|_{\M^{q}(p;B_r)}^{\frac{pq}{pq-N}} \int_{B_r} |\phi|^p \dx.
	\end{align}
	Use the local uniform ellipticity of $A$, \eqref{Req:2}, and \eqref{Req:4} to obtain
	\begin{align*}
	& \, \int_{B_r} |\nabla \phi|_A^p \dx + \int_{B_r} V|\phi|^p \dx -M \int_{B_r} |g| |\phi|^p \dx 
	\geq  \left[\frac{\theta_{B_1}}{2}-\delta \right] \int_{B_r} |\nabla \phi|^p \dx\\
	&   \qquad \qquad   - \frac{C(N,p,q)}{\delta^{\frac{N}{pq-N}}} \|V\|_{\M^{q}(p;B_r)}^{\frac{pq}{pq-N}} \int_{B_r} |\phi|^p \dx   - \frac{MC(N,p,q)}{[\frac{\theta_{B_1}}{2M}]^{\frac{N}{pq-N}}}  \|g\|_{\M^{q}(p;B_r)}^{\frac{pq}{pq-N}} \int_{B_r} |\phi|^p \dx \\
	& \geq \left( C_0 r^{-p}\left[\frac{\theta_{B_1}}{2}-\delta \right]   -\frac{C(N,p,q)}{\delta^{\frac{N}{pq-N}}} \|V\|_{\M^{q}(p;B_1)}^{\frac{pq}{pq-N}}    - \frac{MC(N,p,q)}{\big[\frac{\theta_{B_1}}{2M}\big]^{\frac{N}{pq-N}}}  \|g\|_{\M^{q}(p;B_1)}^{\frac{pq}{pq-N}} \right)  \|\phi\|_{L^p(B_r)}^p,
	\end{align*}
	where the latter inequality follows from \eqref{estimate:ev} and Remark \ref{Rmk:Morrey}-$(ii)$.
	Choosing $0<\delta < \frac{\theta_{B_1}}{2}$, it follows that for any $M>0$ there exists $C_1>0$ such that the right hand side of the above inequality is greater than  $C_1 r^{-p}\|\phi\|_{L^p(B_r)}^p$  for sufficiently small $r\!>\!0$. Hence,  $\S_g(x,\Om)\!=\! \infty.$ 
\end{proof}
\section{Attainment of the best constant II}
In this section, we use concentration compactness arguments to provide another sufficient condition for the best constant in \eqref{WH} to be attained in $\DpAV$.

Throughout this section we assume the following condition on $V$:
$${\bf{(H1)}}: V \in \mathcal{M}_{{\rm loc}}^{q}(p; \Om) \ \text{is such that} \ V^- \in {\mathcal{H}_{p,0}(\Om,V^+)} .$$
By Theorem~\ref{Thm:compactness}, if (H1) holds, then $T_{V^-}$ is compact in $\DpAV$. In addition, by \eqref{WHV3}, we have
$$\frac{1}{\mathcal{B}_g(\Om)}= \S_g(\Om)=\inf \{Q_{p,A,V}(\phi) \bigm| \phi \in \DpAV \,, \int_{\Om} |g| |\phi|^p \dx =1\} \,.$$
 Let $\bigcup_n \Om_n= \Om$ be a compact smooth exhaustion of $\Om$. Recall the definitions of $\S_{g}(x,\Gw)$, $\S_g^*(\Gw)$,  and $\Gs_g$. Motivated by \cite{Smets,Tertikas}, we define 
\begin{align*}
\S_{g}^{\infty}(\Gw)  :=  \lim_{R\ra \infty} \inf \left \{Q_{p,A,V}(\phi) \bigm|  \phi \in \mathcal{D}^{1,p}_{A,V^+} (\Om \cap \overline{B_R}^c), \ \int_{\Om} |g||\phi|^p  \dx =1 \right \} ,\\
\end{align*}
 Clearly, $\S_g(\Om)\leq \min\{\S_g^*(\Gw), \S_g^{\infty}(\Gw)\}$.   We make the following assumption on $g$:
$${\bf{(H2)}}:  g \in \mathcal{H}_p(\Om, V)  \text{ with }   \left|\overline{\Gs_g}\right|=0   .$$
Now we state our result.
\begin{theorem} \label{bestconst_thm}
	Let $g, V$ satisfy {\rm{(H1), (H2)}}  with $\S_g(\Om) < \min\{\S_g^*(\Gw),\S_g^{\infty}(\Gw)\}$. Then $\mathcal{B}_g(\Om)$ is attained in $\DpAV$.
\end{theorem}
In order to prove Theorem \ref{bestconst_thm}, we need a $(g,V)$-depended concentration compactness lemma. For $V = 0$ and $A=I_{N\times N}$, analogous results are obtained in \cite{AU,Smets,Tertikas}.
\subsection{A variant of concentration compactness lemma}
Let $\mathbb{M}(\R^N)$ be the space of all Radon measures (i.e., regular, finite,  Borel signed-measures). Recall that $\mathbb{M} (\R^N)$ is a Banach space with respect to the norm $\norm{\mu}=|\mu|(\R^N)$ (total variation of the measure $\mu$). By the Riesz representation theorem \cite[Theorem 14.14]{Border}, $\mathbb{M}(\R^N)$ is the dual of $\text{C}_0(\R^N)$ := $\overline{\text{C}_c(\R^N)}$ in $L^{\infty}(\R^N)$. A  sequence $(\mu_n)$ is said to be {\em weak* convergent} to $\mu$ in $\mathbb{M}(\R^N)$ if
 \begin{eqnarray*}
  \lim_{n \ra \infty}\int_{\R^N} \phi \dmu_n =  \int_{\R^N} \phi \dmu \qquad     \forall\,   \phi \in \text{C}_0(\R^N),
 \end{eqnarray*}
 and we denote $\mu_n \wrastar \mu$. It follows from the Banach-Alaoglu theorem  that if $(\mu_n)$ is a bounded sequence in $\mathbb{M}(\R^N)$, then (up to a subsequence) there exists $\mu \in \mathbb{M}(\R^N)$ such that $\mu_n \overset{\ast}{\rightharpoonup} \mu$.

 A function in $\DpAV$ can be considered as a function in $\mathcal{D}^{1,p}_{A,V^+}(\R^N)$ by its zero extension.  Following this convention,  for $u_n,u \in \DpAV$ and a Borel set  $ E $ in   $\R^N,$ we  define the following sequences of measures: \begin{align*}
 \nu_n(E):=\int_E |g||u_n - u|^p \dx, & \qquad \Ga_n(E):=\int_E(|\nabla(u_n-u)|_A^p + V |u_n-u|^p)\dx , \\
 \widetilde{\Ga}_n(E) & :=\int_E (|\nabla u_n|_A^p + V |u_n|^p)\dx \,,
 \end{align*}
 where $g \in \mathcal{H}_p(\Om,V)$.
If $u_n\wra u$ in $\DpAV$, then by assumptions (H1) and (H2), the sequences $\nu_n$, $\Gamma_n$, and $\widetilde{\Gamma}_n$
  have weak* convergent subsequences. Let
  \[\nu_n \wrastar \nu; \qquad \Gamma_n \overset{\ast}{\rightharpoonup} \Gamma; \qquad \widetilde{\Gamma}_n \overset{\ast}{\rightharpoonup} \widetilde{\Gamma}  \quad \mbox{ in } \mathbb{M}(\R^N).\]
\begin{proposition} \label{Prop:ineq}
 Let $\vge  >0$. Then, for $p \in (1,\infty)$, there exists $C(\vge ,p)>0$ such that
  $$||a+b|_A^p-|a|_A^p| \leq \vge  |a|_A^p + C(\vge ,p) |b|_A^p \qquad   \forall a,b \in \R^N.$$ 
\end{proposition} 
\begin{proof}
Let $\vge  >0$ and $p \in (1,\infty)$. Then there exists   $C(\vge ,p)>0$ such that 
\begin{equation} \label{inequality}
\left||s+t|^p-|s|^p \right|\leq \vge  |s|^p + C(\vge ,p)|t|^p \qquad   \forall s,t \in \R
\end{equation} 
(see \cite{Lieb}, page 22). Now, for any $\theta \in [-1,1]$,  using \eqref{inequality} we obtain
  $$||s^2+2\theta s+1|^{\frac{p}{2}} -s^p| \leq \vge  |s|^p + C(\vge ,p) \qquad   \forall s \in \R$$
for some $C(\vge ,p)>0$. By taking $s=\frac{|a|_A}{|b|_A}$ and $\theta = \frac{<A(x) b,a>}{|a|_A|b|_A}$ we obtain our claim.
\end{proof}
\begin{lemma} \label{uineq}
  Let $\Phi \in C_b^1(\Om)$ be such that $\nabla \Phi$ has compact support and $(u_n) \in W^{1,p}(\Om) \cap C_c(\Om)$ be such that $u_n \wra u$ in $\DpAV$. Then
  \[ \lim_{n \ra \infty} \int_{\Om} |\nabla((u_n -  u)\Phi)|_A^p \dx = \lim_{n \ra \infty} \int_{\Om} |\nabla(u_n -  u)|_A^p |\Phi|^p \dx.\]
 \end{lemma}
 \begin{proof} 
  Let $\vge  >0$ be given. Using Proposition \ref{Prop:ineq},
  \begin{eqnarray*}
      &\bigg| \int_{\Om} |\nabla((u_n  -  u)\Phi)|_A^p \dx - \int_{\Om} |\nabla(u_n -  u)|_A^p |\Phi|^p  \dx\bigg|\\
               &\leq   \vge \int_{\Om} |\nabla(u_n -  u)|_A^p |\Phi|^p \dx + C(\vge,p) \int_{\Om} |u_n -  u|^p |\nabla \Phi|_A^p \dx. 
  \end{eqnarray*}
 Since $\nabla \Phi$ is compactly supported, it follows from a similar arguments as in Remark \ref{Remark:cpct}-$(iii)$ that $\int_{\Om} |u_n -  u|^p |\nabla \Phi|_A^p \dx \to 0$ as $n \ra \infty$. Further, as $(u_n)$ is bounded in $\DpAV$ and $\vge>0$ is arbitrary, we obtain the desired result.
 \end{proof}
   Next, we prove the absolute continuity of the measure $\nu$ with respect to $\Gamma$. 
  \begin{lemma}\label{lem:abs_cont}
 Let $V$ satisfy {\rm{(H1)}}, $g \in \mathcal{H}_p(\Om,V)$, and $u_n \in W^{1,p}(\Om) \cap C_c(\Om)$ be such that  $u_n \wra u$ in $\DpAV$.  Then,  for any Borel set $E$ in $\R^N$,
  \[\S^*_g \ \nu(E) \leq \Gamma (E), \ \mbox{where $\S^*_g = \displaystyle\inf_{x \in \overline{\Om}} \S_g(x)$ } .\]
   In particular, $\nu$ is supported on $\overline{\Sigma_g}$.
  \end{lemma}
\begin{proof}
  For $\Phi \in C_c^{\infty}(\R^N)$, $(u_n-u) \Phi \in \DpAV$. Therefore, since $g \in \H_p(\Om,V)$, using Remark \ref{Remark_ref}, we obtain
\begin{align} \label{req:1}
 \int_{\R^N} \!\!|\Phi|^p \!\!\dnu_n  & \!=\! \int_{\Om}\! |g||(u_n-u)\Phi|^p \dx   \leq   
 \mathcal{B}_g(\Om)  \!  \int_{\Om} \!\left[|\nabla((u_n -  u)\Phi)|_A^p +V|(u_n -  u)\Phi|^p \right] \dx \nonumber\\ 
 & =  \mathcal{B}_g(\Om)    \int_{\R^N} \left[|\nabla((u_n -  u)\Phi)|_A^p +V|(u_n -  u)\Phi|^p \right] \dx.
\end{align}
By Proposition \ref{Prop:ineq}  \[ \lim_{n \ra \infty} \int_{\R^N} |\nabla((u_n -  u)\Phi)|_A^p \dx = \lim_{n \ra \infty} \int_{\R^N} |\nabla(u_n -  u)|_A^p |\Phi|^p \dx.\]
Thus, by taking $n \ra \infty$ in \eqref{req:1}, we obtain
 \begin{align*} \label{forrmk}
  \int_{\R^N} |\Phi|^p \dnu \leq \mathcal{B}_g(\Om)  \int_{\R^N} |\Phi|^p \dGamma .
 \end{align*}
Thus, 
  \begin{equation*}\label{measureinequality1}
    \nu(E) \leq \mathcal{B}_g(\Om)   \Gamma (E)  \qquad \forall E \mbox{ Borel  set in }\R^N.
  \end{equation*}
 In particular, $\nu \ll \Gamma$ and hence by Radon-Nikodym theorem, 
 \begin{equation} \label{measureinequality}
  \nu(E) = \int_E  \frac{\dnu}{\dGamma}  \dGamma \qquad \forall E \mbox{ Borel  in }\R^N. 
 \end{equation}
 Further, by Lebesgue differentiation theorem  \cite{ Federer}, we obtain 
 \begin{equation} \label{Lebdiff}
  \frac{\d \nu}{\d \Gamma}(x) = \lim_{r \ra 0} \frac{\nu (B_r(x))}{\Gamma (B_r(x))}.
 \end{equation}
 Now replacing $g$ by $g \chi_{B_r(x)}$ and proceeding as before,
 \[   \nu(B_r(x)) \leq  \mathcal{B}_g(\Om \cap B_{r}(x)) \ \Gamma (B_r(x)).\] 
 Thus from \eqref{Lebdiff} we get 
\begin{eqnarray} \label{21}
 \frac{\dnu}{\dGamma} (x) \leq \frac{1}{\S_g(x)} \leq \frac{1}{\S_g^*}\, .
\end{eqnarray} 
Hence, 
$\norm{\frac{\dnu}{\dGamma}}_{\infty} \leq \frac{1}{\S_g^*}\,$, and using \eqref{measureinequality},  $\S_g^* \ \nu(E) \leq  \Gamma (E)$ 
for all Borel subsets $E$ of $\R^N$.
\end{proof}

The next lemma gives a lower estimate for the measure $\tilde{\Gamma}.$ For $V = 0$ and $A=I_{N \times N}$, similar estimate is obtained in  \cite[Lemma 2.1]{Smets}.  
\begin{lemma} \label{lemma:esstimate}
Let $V,g $ satisfy {\rm{(H1), (H2)}}. If $u_n \in W^{1,p}(\Om) \cap C_c(\Om)$ is such that $u_n \wra u$ in $\DpAV$, then 
$ \widetilde{\Gamma} \geq
     |\nabla u|_A^p + V|u|^p + \S_g^*\nu .$
\end{lemma}

\begin{proof}
Our proof splits into three steps.

{\bf Step 1:} $\tilde{\Ga} \geq |\nabla u|_A^p + V|u|^p$. Indeed, let $\phi \in C_c^{\infty}(\R^N)$ with $0\leq \phi \leq 1$, we need to show that
$\int_{\R^N} \phi \dtildeGa \geq \int_{\R^N} \phi \left[|\nabla u|_A^p+ V|u|^p \right] \dx. $
Notice that,
\begin{align*}
    \int_{\R^N} \phi \dtildeGa= \lim_{n \ra \infty} \int_{\R^N} \phi \dtildeGa_n  
    & = \lim_{n \ra \infty} \int_{\Om} \phi \left[|\nabla u_n|_A^p+ V|u_n|^p \right] \dx \\
     & = \lim_{n \ra \infty} \int_{\Om} \left[F(x, u_n(x),\nabla u_n(x)) -V^- |u_n|^p \right] \dx ,
\end{align*}
where $F:\Om \times \R \times \R^N \to \R$ is defined as $F(x,r,z)=\phi(x)[|z|_{A(x)}^p+V^+(x) r^p].$ Clearly, $F$ is a Carath\'eodory function  and $F(x,r,\cdot)$ is convex for almost every $(x,r) \in \Om \times \R$. Hence, by 
\cite[Theorem 2.11]{Fillip}, we have
\begin{align*}
    \int_{\R^N} \phi \left[|\nabla u|_A^p + V^+ |u|^p \right] \dx &= \int_{\Om} \phi \left[|\nabla u|_A^p + V^+ |u|^p \right] \dx \\ 
    & \leq\liminf_{n \ra \infty} \int_{\Om} \phi \left[|\nabla u_n|_A^p + V^+ |u_n|^p \right] \dx .
\end{align*} 
On the other hand,  $V^- \in \mathcal{H}_{p,0}(\Om,V^+)$ implies that $V^-\phi \in \mathcal{H}_{p,0}(\Om,V^+)$. Hence, by Theorem \ref{Thm:compactness}  we have $\lim_{n \ra \infty} \int_{\Om} V^- \phi  |u_n|^p \dx = \int_{\Om} V^- \phi  |u|^p \dx$. Combining these two facts we  obtain Claim 1.

\noi {\bf Step 2:} $\tilde{\Gamma}=\Gamma$ on $\overline{\Gs_g}$. Indeed, let $E\subset\overline{\Gs_g}$ be a Borel set. Thus, for each $m \in \N$, there exists an open subset $O_{m}$ containing $E$ such that $|O_m|=|O_{m} \setminus E| < \frac{1}{m}$. 
Let $\vge >0$ be given. Using Lemma~\ref{uineq} it follows that for any $\phi \in C_c^{\infty}(O_{m})$ with $0 \leq \phi \leq 1$, we have
\begin{align*}
  & \left|\int_{\Om}  \phi \dGamma_n -\int_{\Om}  \phi \dtildeGamma_n  \right|  
  = \left|\int_{\Om}  \phi \left[|\nabla (u_n-u)|_A^p + V |u_n-u|^p \right] \dx -\int_{\Om}  \phi \left[|\nabla u_n|_A^p + V |u_n|^p \right] \dx \right| \\[2mm]
     &\!\leq \!\!  \int_{\Om} \!\! \phi \left||\nabla (u_n-u)|_A^p-|\nabla u|_A^p \right| \!\dx \!+ \!\!\!\int_{\Om} \phi V^+ \!\left||u_n-u|^p-|u_n|^p\right| \! \dx   \!  + \!\! \int_{\Om} \!\! \phi   V^- \left||u_n-u|^p-|u_n|^p \right|  \!\dx  \\[2mm]
    &\leq \left[\vge \int_{\Om}  \phi |\nabla u_n|_A^p \dx + \text{C}_1(\vge,p) \int_{\Om} \phi |\nabla u|_A^p \dx \right] + \left[\vge \int_{\Om}  \phi V^+| u_n|^p \dx + \text{C}_2(\vge,p) \int_{\Om} \phi V^+| u|^p \dx \right] \\[2mm]
    & + \left[\vge \int_{\Om}  \phi V^-| u_n|^p \dx + \text{C}_2(\vge,p) \int_{\Om} \phi V^-| u|^p \dx \right]  \leq   \vge \text{C}_4 L + \text{C}_3(\vge,p) \int_{O_{m}}   [|\nabla u|_A^p+V^+|u|^p] \dx,
\end{align*}
where $L=\sup_{n}\left\{\int_{\Om} [|\nabla u_n|_A^p+V^+|u_n|^p] \dx\right\}$, and in the latter inequality we used the fact that $V^- \in \mathcal{H}_p(\Om, V^+)$.
Now letting $n \ra \infty$, we obtain 
$$
 \left|\int_{\Om}  \phi \dGamma  -\int_{\Om}  \phi \dtildeGamma  \right|  \leq \vge C_4 L + \text{C}(\vge,p) \int_{O_{m}}   [|\nabla u|_A^p+V^+|u|^p]\dx.
$$
Therefore,
\begin{eqnarray*} 
\left|\Gamma (O_m)-\tilde{\Gamma} (O_m) \right| &=& \sup \left\{ \left|\int_{\Om}  \phi \dGamma -\int_{\Om}  \phi \dtildeGamma  \right| \, \bigm|\, \phi \in C_c^{\infty}(O_m), 0 \leq \phi \leq 1 \right \} \\
&\leq & \vge C_4L + \text{C}(\vge,p) \int_{O_{m}}   [|\nabla u|_A^p+V^+|u|^p]\dx,
\end{eqnarray*}
Now as $m \ra \infty,$ $|O_m|\ra 0$ (since $\left|\overline{\Gs_g}\right|=0 $), and hence, $| \Gamma (E)-\tilde{\Gamma} (E)| \leq \vge C_4L$. Since $\vge >0$ is arbitrary, we conclude $\Gamma(E)=\tilde{\Gamma} (E).$

\noi {\bf{Step 3:}} $ \tilde{\Gamma} \geq |\nabla u|_A^p + V|u|^p+ \S_g^*\nu$. From Lemma \ref{lem:abs_cont} we have  $ \Gamma \geq  \S_g^*\nu$. Furthermore,  by Lemma~\ref{lem:abs_cont}, $\nu$ is supported on $\overline{\Gs_g}$. Hence Step 1 and Step 2 yield the following: 
\begin{equation}\label{rep1}
 \tilde{\Gamma} \geq \left\{\begin{array}{ll}
    |\nabla u|_A^p +V|u|^p,  &  \\
      \S_g^*\nu. &
\end{array}\right.    
\end{equation}
Since $\left|\overline{\Gs_g}\right|=0$, the measure $|\nabla u|_A^p+V|u|^p$ is supported in $\overline{\Gs_g}^{c}$, and hence from \eqref{rep1} we easily obtain  $\tilde{\Gamma} \geq |\nabla u|_A^p + V|u|^p + \S_g^*\nu.$
\end{proof}

\begin{lemma}\label{lemmasmets}
 Suppose that $V$ satisfies {\rm{(H1)}}, $g \in \mathcal{H}_p(\Om,V)$ is nonnegative. Let $u_n \in W^{1,p}(\Om) \cap C_c(\Om)$ be such that $u_n \wra u$ in $ \DpAV$, and $\Phi_R \in C_b^{\infty}(\R^N)$ with $0\leq \Phi_R \leq 1$, $\Phi_R = 0 $ on $\overline{B_R}$ and $\Phi_R = 1 $ on $B_{R+1}^c$.
Then,
\begin{align*}
&(A) \,    \lim_{R \ra \infty} \overline{\lim_{n \ra \infty}}  \int_{\Om \cap \overline{B_R}^c} g|u_n|^p \dx =\lim_{R \ra \infty}   \overline{\lim_{n \ra \infty}}    \nu_n(\Om \cap \overline{B_R}^c) =\lim_{R \ra \infty} \overline{\lim_{n \ra \infty}} \int_{\Om}  \Phi_R \dnu_n, \\
&(B) \,     \lim_{R \ra \infty}\overline{\lim_{n \ra \infty} } \!\!  \int_{\Om \cap \overline{B_R}^c} \!\left[
|\nabla u_n|^p \! + \!V  |u_n|^p  \right]\!  \dx  \!=
\!  \lim_{R \ra \infty} 
\overline{\lim_{n \ra \infty}}  \Gamma_n(\Om \cap \overline{B_R}^c)\! = 
\!\lim_{R \ra \infty} \overline{\lim_{n \ra \infty}} \! \int_{\Om}  \!\!\Phi_R \dGamma_n.
\end{align*}    
 \end{lemma} 

\begin{proof}
 By Brezis-Lieb lemma~\cite[Theorem~1.9]{Lieb},
 \begin{align*}
 &\left|\displaystyle \overline{\lim_{n \ra \infty}}  \nu_n(\Om \cap \overline{B_R}^c) - \displaystyle \overline{\lim_{n \ra \infty}}\int_{\Om \cap \overline{B_R}^c} g|u_n|^p \dx \right|\leq 
 \displaystyle \overline{\lim_{n \ra \infty}} \left| \nu_n(\Om \cap \overline{B_R}^c) - \int_{\Om \cap \overline{B_R}^c} g|u_n|^p \dx \right| \\
& = 
 \overline{\lim_{n \ra \infty}} \left|\int_{\Om \cap \overline{B_R}^c} g|u_n-u|^p \dx - \int_{\Om \cap \overline{B_R}^c} g|u_n|^p \dx \right| = \int_{\Om \cap \overline{B_R}^c} g|u|^p \dx.
  \end{align*}
 By \eqref{WHV3}, $g|u|^p \in L^1(\Om)$, therefore, the right-hand side integral goes to $0$ as $R \ra \infty$. Thus, we get the first equality in $(A)$. For the second equality, it is enough to observe that
 \begin{eqnarray*}
  \int_{\Om \cap \overline{B_{R+1}}^c} g|u_n -u|^p \dx \leq \int_{\Om} g|u_n -u|^p \Phi_R \dx \leq \int_{\Om \cap \overline{B_{R}}^c} g|u_n -u|^p \dx.
 \end{eqnarray*}
Now by taking $n\to \infty $, and then $R \ra \infty$, we get the required equality. 

Now we proceed to prove (B). For $\vge >0$, using Lemma \ref{uineq} we estimate
\begin{align*}
& \displaystyle \overline{\lim_{n \ra \infty}} \left| \Gamma_n(\Om \cap \overline{B_R}^c) - \int_{\Om \cap \overline{B_R}^c} \left[|\nabla u_n|_A^p+V|u_n|^p \right] \dx \right| \\
&= 
\overline{\lim_{n \ra \infty}} \left|\int_{\Om \cap \overline{B_R}^c} \left[|\nabla (u_n-u)|_A^p+V|u_n-u|^p \right] \dx - \int_{\Om \cap \overline{B_R}^c} \left[|\nabla u_n|_A^p+V|u_n|^p \right] \dx \right| \\
& \leq  
\overline{\lim_{n \ra \infty}} \left(\left|\int_{\Om \cap \overline{B_R}^c} [|\nabla (u_n-u)|_A^p  - |\nabla u_n|_A^p] \dx \right| +  \left|\int_{\Om \cap \overline{B_R}^c} [V|u_n-u|^p  - V|u_n|^p]  \dx \right| \right)  
\\& = \left( \vge \overline{\lim_{n \ra \infty}} \int_{\Om \cap \overline{B_R}^c}   |\nabla u_n|_A^p \dx + \text{C}_1(\vge,p) \int_{\Om \cap \overline{B_R}^c} |\nabla u|_A^p \dx \right) 
\\ &     \qquad \qquad +  \left( \vge \overline{\lim_{n \ra \infty}} \int_{\Om \cap \overline{B_R}^c}   (V^++V^-)|u_n|^p \dx + \text{C}_2(\vge,p) \int_{\Om \cap \overline{B_R}^c}  (V^++V^-) |u|^p \dx \right) \\
&  \leq  \left( \vge \text{C}_3  L  + \text{C}_4(\vge,p) \int_{\Om \cap \overline{B_R}^c} [|\nabla u|_A^p+V^+|u|^p] \dx  \right) ,
\end{align*}
where $L=\sup_{n}\left\{\int_{\Om} [|\nabla u_n|_A^p+V^+|u_n|^p] \dx\right\}$. The latter inequality uses once again the fact that $V^- \in \mathcal{H}_p(\Om, V^+)$.

Thus, due to the subadditivity of the limsup, and by taking $R \ra \infty$ and then $\vge \ra 0$, we obtain the first equality of (B). The second equality of part (B) follows from the same argument as the corresponding part of (A). 
\end{proof} 
\begin{lemma}\label{mlc2}
  Let $V$ satisfy {\rm{(H1)}}, $g \in \mathcal{H}_p(\Om,V)$ be nonnegative,  and $u_n \in W^{1,p}(\Om) \cap C_c(\Om)$ be such that  $u_n \wra u$ in $\DpAV$. Set
  \begin{eqnarray*}
    \nu_{\infty} = \displaystyle\lim_{R \ra \infty} \overline{\lim_{n \ra \infty}}  \nu_n(\Om \cap \overline{B_R}^c) \quad \mbox{and} \quad \Gamma_{\infty} =  \displaystyle\lim_{R \ra \infty} \overline{\lim_{n \ra \infty}}  \Gamma_n(\Om \cap \overline{B_R}^c).
  \end{eqnarray*} 
  Then
\begin{enumerate}[(i)]
 \item $ \S_g^{\infty} \nu_{\infty} \leq   \Gamma_{\infty} \nonumber,$
  \item $ \overline{\lim}_{n \ra \infty} \int_{\Om} g|u_n|^p \dx = \int_{\Om} g|u|^p \dx + \norm{\nu} + \nu_{\infty}.$
\item[(iii)] Further, if $|\overline{\Gs_g}|=0$, then we have

\begin{equation*}
\displaystyle \overline{\lim_{n \ra \infty}} \int_{\Om} [|\nabla u_n|_A^p+ V|u_n|^p] \dx  \geq 
\displaystyle \int_{\Om} [|\nabla u|_A^p + V|u|^p] \dx + \S_g^*\norm{\nu} +  \Gamma_{\infty} \,.
  \end{equation*}
\end{enumerate}
  \end{lemma}
\begin{proof}
 (i): For $R>0$, choose $\Phi_R \in C_b^{1}(\R^N)$ satisfying  $0\leq \Phi_R \leq 1$, $\Phi_R = 0 $ on $\overline{B_{R+\frac{1}{2}}}$ and $\Phi_R = 1 $ on $B_{R+1}^c$. Clearly, $(u_n-u) \Phi_R \in \D^{1,p}_{A,V^+}( \Om \cap \displaystyle\overline{B_R}^c)$, and since $g \in \mathcal{H}_p(\Om, V)$, we have
\[ \int_{\Om \cap \overline{B_R}^c} \!\! g|(u_n-u)\Phi_R|^p \dx \leq \! \mathcal{B}_{g}(\Om \cap \overline{B_R}^c)  \!\left[\int_{\Om \cap \overline{B_R}^c} \!\!|\nabla((u_n -  u)\Phi_R)|_A^p \dx + \!\int_{\Om \cap \overline{B_R}^c} \!\!V|u_n -  u|^p \dx\right]. \]
By Lemma \ref{uineq}  
$$\displaystyle \lim_{n \ra \infty} \int_{\Om \cap \overline{B_R}^c} |\nabla((u_n -  u)\Phi_R)|_A^p \dx =  \lim_{n \ra \infty} \int_{\Om \cap \overline{B_R}^c}   |\nabla(u_n -  u)|_A^p\Phi_R^p \,\dx.$$ 
Therefore, letting $n \ra \infty$, $ R \ra \infty$, and using Lemma \ref{lemmasmets} successively in the above inequality, we obtain $ \S_g^{\infty}\nu_{\infty} \leq  \Gamma_{\infty}.$ 

\noi $(ii)$: Choosing $\Phi_R$ as in Lemma \ref{lemmasmets} and using Brezis-Lieb lemma \cite[Theorem~1.9]{Lieb}, we have
\begin{eqnarray}
 \overline{\lim_{n\ra \infty}} \int_{\Om} g| u_n|^p \dx &=&  \overline{\lim_{n\ra \infty}} \left[ \int_{\Om} g| u_n|^p (1-\Phi_R)\dx + \int_{\Om} g| u_n|^p \Phi_R \dx \right] \nonumber \\
                                          &=& \overline{\lim_{n\ra \infty}} \!\left[\!  \int_{\Om} g| u|^p (1-\Phi_R)\dx +\!\! \int_{\Om} g| u_n-u|^p (1-\Phi_R) \dx + \!\!\int_{\Om} g| u_n|^p \Phi_R \dx\right]   \nonumber \\
                                          &=&    \int_{\Om} g| u|^p (1-\Phi_R)\dx + \int_{\Om}  (1-\Phi_R) \d\nu + \overline{\lim_{n\ra \infty}} \int_{\Om} g| u_n|^p \Phi_R \dx  \label{need1} 
  \end{eqnarray} 
The last equality uses the facts that $g|u_n-u|^p \wrastar \nu$, and $\displaystyle\overline{\lim_{n\ra \infty}} (a_n+b_n) = \lim_{n \ra \infty}a_n + \overline{\lim_{n\ra \infty}} b_n$ for any real sequences $(a_n), (b_n)$ with $(a_n)$ being convergent. Now, by taking $R \ra \infty$ in \eqref{need1} and using Lemma \ref{lemmasmets}-$(A)$, we obtain
\begin{align*}
\overline{\lim_{n\ra \infty}} \int_{\Om} g| u_n|^p \dx = \int_{\Om} g| u|^p \dx + \norm{\nu} + \nu_{\infty}. 
\end{align*} 
 
 \noi $(iii)$: Notice that 
 \begin{align*}
&\, \overline{\lim_{n\ra \infty}} \int_{\Om} [|\nabla u_n|_A^p + V|u_n|^p]\dx \\ &= \overline{\lim_{n\ra \infty}} \left[ \int_{\Om} [|\nabla u_n|_A^p + V|u_n|^p] (1-\Phi_R) \dx + \int_{\Om} [|\nabla u_n|_A^p + V|u_n|^p] \Phi_R \right] \dx \\ & = \, \, \int_{\Om} (1-\Phi_R)\tilde{\Gamma} + \overline{\lim_{n\ra \infty}} \int_{\Om} [|\nabla u_n|_A^p + V|u_n|^p] \Phi_R  \dx .
  \end{align*}
The last equality uses the facts that $|\nabla u_n|^p + V|u_n|^p \wrastar \widetilde{\Gamma}$, and $\displaystyle\overline{\lim_{n\ra \infty}} (a_n+b_n) = \lim_{n \ra \infty}a_n + \overline{\lim_{n\ra \infty}} b_n$ for any real sequences $(a_n), (b_n)$ with $(a_n)$ being convergent. By taking $R \ra \infty$ and using part (B) Lemma \ref{lemmasmets} we get 
\begin{equation*} 
 \overline{\lim_{n\ra \infty}} \int_{\Om} [|\nabla u_n|_A^p + V|u_n|^p] \dx= \|\tilde{\Gamma}\|+ \Gamma_{\infty}.   
\end{equation*}
Now, using Lemma \ref{lemma:esstimate}, we obtain 
\begin{equation*} 
\displaystyle \overline{\lim_{n\ra \infty}} \int_{\Om} [|\nabla u_n|_A^p + V|u_n|^p] \dx \geq 
\displaystyle \int_{\Om} [|\nabla u|_A^p + V|u|^p] \dx + \S_g^*\norm{\nu} +  \Gamma_{\infty} \,. \qedhere
  \end{equation*}
 \end{proof}
	\begin{proof}[Proof of Theorem \ref{bestconst_thm}] Recall that, for $g \in \mathcal{H}_p(\Om,V)$,   the best constant $\mathcal{B}_g(\Om)$ in \eqref{WH} is given by the Rayleigh-Ritz variational principle 
\begin{align*}
    \frac{1}{\mathcal{B}_g(\Om)} & = \inf \left\{\mathcal{R}_V(\phi):=\frac{\int_{\Om} [|\nabla \phi|_A^p+V|\phi|^p] \dx}
    {\int_{\Om} |g| |\phi|^p \dx} \,  \biggm|  \,   \phi \in W^{1,p}(\Om) \cap C_c(\Om) \setminus \{0\}  \right\} \\
    &= \displaystyle \inf \left\{\mathcal{R}_V(\phi) \bigm|  \phi \in W^{1,p}(\Om) \cap C_c(\Om), \, \int_{\Om} [\nabla\phi|_A^p+V^+|\phi|^p] \dx =1   \right\},
\end{align*} 
where in the latter equality we used the homogeneity of the Rayleigh quotient.
Let $(u_n)$ be a sequence that minimizes $ \mathcal{R}_V(\phi)$ over $\{\phi\mid \phi \in W^{1,p}(\Om) \cap C_c(\Om) \,, \int_{\Om}  [|\nabla \phi|_{{A}}^p+V^+|\phi|^p =1\}$. Then, one can see that
\begin{align} \label{1st_estimate}
\overline{\lim}_{n \ra \infty} \int_{\Om}|g||u_n|^p \dx  \geq \mathcal{B}_g(\Om)  \, \underline{\lim}_{n \ra \infty} \int_{\Om} [|\nabla u_n|_A^p+V|u_n|^p] \dx .
\end{align}
Since $\displaystyle \int_{\Om}  [|\nabla u_n|_{{A}}^p+V^+|u_n|^p]\dx =1$ for all $n \in \N$, i.e., $(u_n)$ is bounded in $\DpAV$, it follows that $ u_n \wra u $ in $ \DpAV$ (up to a subsequence). $V^- \in \mathcal{H}_{p,0}(\Om,V^+)$ implies that $\displaystyle \lim_{n\ra \infty} \int_{\Om} V^-|u_n|^p \dx=\int_{\Om} V^-|u|^p \dx$ (by Remark \ref{Rmk:Hptilde}-$(i)$). Thus, we obtain
\begin{align} \label{limsup_liminf}
\lim_{n \ra \infty} \int_{\Om} [|\nabla u_n|_A^p+V|u_n|^p] \dx
 & = \lim_{n \ra \infty} \int_{\Om} [|\nabla u_n|_A^p+V^+|u_n|^p] \dx - \lim_{n \ra \infty} \int_{\Om} V^-|u_n|^p \dx 
\end{align}
Recall our assumptions that  $ u_n \wra u $ in $ \DpAV$ and that up to a subsequence,  
$$ |\nabla u_n - \nabla u|^p + V| u_n -  u|^p \wrastar \Gamma, \ \ |\nabla u_n|^p+ V| u_n |^p \wrastar \tilde{\Gamma}, \ \ |g||u_n - u|^p \wrastar \nu  \ \mbox{in } \mathbb{M}(\R^N).$$ 
Using  Lemma \ref{mlc2} we get 
\begin{equation}\label{eq-gun}
\overline{\lim}_{n \ra \infty} \int_{\Om} |g| |u_n|^p \dx= \int_{\Om} |g| |u|^p \dx+ \norm{\nu} + \nu_{\infty} . \end{equation}

Suppose that $\norm{\nu}+\nu_{\infty} \neq 0$.  Recalling that the Hardy-type inequality holds for $u \in \DpAV$ (by Remark~\ref{Remark_ref}-$(i)$), and using \eqref{limsup_liminf}, Lemma~\ref{mlc2} $(iii)$, the spectral gap assumption, and  finally \eqref{eq-gun} , we obtain
\begin{align*}
&\overline{\lim}_{n \ra \infty} \int_{\Om}|g||u_n|^p \dx  \geq \mathcal{B}_g(\Om)  \, \underline{\lim}_{n \ra \infty} \int_{\Om} [|\nabla u_n|_A^p+V|u_n|^p] \dx  \\ 
& \geq \mathcal{B}_g(\Om)  \, \overline{\lim}_{n \ra \infty} \int_{\Om} [|\nabla u_n|_A^p+V|u_n|^p] \dx  
\geq  \mathcal{B}_g(\Om) \left( \int_{\Om} [|\nabla u|_A^p+V|u|^p] \dx +  \S_g^*\norm{\nu} + \Gamma_{\infty} \right) \\
& \geq \mathcal{B}_g(\Om) \left( \frac{1}{\mathcal{B}_g}(\Om) \int_{\Om} |g||u|^p \dx +  \S_g^*\norm{\nu} + \S_g^{\infty}\nu_{\infty} \right) \\
& > \frac{\mathcal{B}_g(\Om)}{\mathcal{B}_g(\Om)} \left(\int_{\Om} |g||u|^p \dx  +  \norm{\nu} + \nu_{\infty}\right)  = \overline{\lim}_{n \ra \infty} \int_{\Om} |g||u_n|^p \dx ,
\end{align*}
which is a contradiction. Thus $\norm{\nu}=\nu_{\infty}=0$.
Therefore, 
$\displaystyle\lim_{n \ra \infty} \int_{\Om}|g||u_n|^p\dx=\int_{\Om} |g||u|^p \dx$. Further, since $u_n \wra u$ in $\DpA$, we have 
$$\int_{\Om} [|\nabla u|_A^p + V^+|u|^p] \dx\leq   \underline{\lim}_{n\ra \infty} \int_{\Om} [|\nabla u_n|_A^p + V^+|u_n|^p] \dx.$$ 
Also, the assumption $V^- \in \mathcal{H}_{p,0}(\Om,V^+)$ implies that $\lim_{n\ra \infty} \int_{\Om} V^-|u_n|^p \dx=\int_{\Om} V^-|u|^p \dx$ (by Remark \ref{Rmk:Hptilde}-$(iii)$). Consequently, $\mathcal{R}(u)$, the Rayleigh-Ritz quotient of $u$, satisfies  
$$\frac{1}{\mathcal{B}_g(\Om)} \leq \mathcal{R}(u) \leq \frac{\underline{\lim}_{n \ra \infty}\int_{\Om} [|\nabla u_n|_A^p+V|u_n|^p] \dx}
{\lim_{n \ra \infty}\int_{\Om} |g| |u_n|^p \dx} \leq \lim_{n \ra \infty} \mathcal{R}(u_n) = \frac{1}{\mathcal{B}_g(\Om)} \,.$$
Hence, $\mathcal{B}_g(\Om)$ is attained at $u$.  
\end{proof}
\begin{remark} \rm
Theorems~\ref{Thm:best_constant} and \ref{bestconst_thm} provide different sufficient conditions for the attainment of the best constant of Hardy-type equalities for $Q_{p,A,V}$. Let us compare these conditions. In Theorem~\ref{bestconst_thm}, we assume that $V \in \mathcal{M}^{q}_{\loc}(p;\Om)$ with $V^- \in \mathcal{H}_{p,0}(\Om,V^+)$. 
Consider for example the operator $-\Delta_p(u) -\gl |x|^{-p}$ in $\Gw= \R^N\setminus \{0\}$ with $0<\gl< ((p-1)/p)^p$, and take  $0\leq g \in C_c^{\infty}(\Om)$. Then $V =-V^-=-\gl |x|^{-p} \not\in  \mathcal{H}_{p,0}(\Om,0)$ but $V_-$  satisfies the integrability condition $\int_{\Gw \setminus K}  V^- G^p\dx<\infty$, where $G$ is a positive solution the equation $(-\Delta_p(u) -\gl |x|^{-p})[u]=0$  of minimal growth in a neighborhood of infinity in $\Gw$. Hence, $V$ satisfies the condition of Theorem~\ref{Thm:best_constant} but not of Theorem~\ref{bestconst_thm}.  On the other hand, in view of Lemma~\ref{lem-Gsg-0},  if $g \in \mathcal{M}^{q}_{\loc}(p;\Om)$, then $\Gs_g \cap \Om =\phi$. Thus, Theorem~\ref{bestconst_thm}  allows stronger local singularities  of $g$ than considered on Theorem~\ref{Thm:best_constant}.
\end{remark}
Next, we characterize the Hardy-weights $g$ such that  $T_g$ is compact on $\DpAV$ in the case  $V=V^+\geq 0$.
\begin{theorem} \label{Cpct_Char}
Let $V $ be nonnegative, and $g \in \mathcal{H}_p(\Om,V)$, $g\neq 0$. Then $T_g$ is compact on $\mathcal{D}^{1,p}_{A,V}$ if and only if $\mathbb{S}_g^*=\mathbb{S}_g^{\infty} =\infty$. 
\end{theorem}
\begin{proof}
Assume that  $T_g$ is compact on $\mathcal{D}^{1,p}_{A,V} (\Om)$, and on the contrary there exists $x \in \overline{\Om}$ such that $\mathbb{S}_g(x)<\infty$. We may assume that $|g|>0$ in a neighborhood of $\Om \cap B_{{1}/{n_0}}(x)$. By the definition of $\mathbb{S}_g(x)$, 
 for each $n \in \N$  large enough there exists $\phi_n \in \mathcal{D}^{1,p}_{A,V} (\Om \cap B_{{1}/{n}}(x))$ such that
$$Q_{p,A,V}(\phi_n) < \mathbb{S}_g(x) + \frac{1}{n}\, ,  \quad \mbox{and } \int_{\Om \cap B_{{1}/{n}}(x)} |g||\phi_n|^p \dx =1.$$
This implies that $(\phi_n)$ is a bounded sequence in $\mathcal{D}^{1,p}_{A,V} (\Om)$, and hence, $\phi_n \wra \phi$ in $\mathcal{D}^{1,p}_{A,V} (\Om)$ (up to a subsequence). Clearly, $\phi =0$ as the supports of $\phi_n$ are shrinking to a singleton set $\{x\}$. Thus, by the compactness of $T_g$, it follows that $\lim_{n \ra \infty} \int_{\Om} |g||\phi_n|^p \dx =0$, which is a contradiction. Therefore, $\mathbb{S}_g(x)=\infty$. Since $x \in \overline{\Om}$ is arbitrary, we infer that $\mathbb{S}_g^{*}=\infty$. Following a similar arguments, one shows that $\mathbb{S}_g^{\infty}=\infty$. 

Conversely, assume that $g \in \mathcal{H}_p(\Om,V)$ satisfies $\mathbb{S}_g^{*} =\mathbb{S}_g^{\infty} =\infty$. Let $(\phi_n)$ be a sequence in $\mathcal{D}^{1,p}_{A,V} (\Om)$ such that $\phi_n \wra \phi$ in $\mathcal{D}^{1,p}_{A,V} (\Om)$. Then, using Lemma \ref{lem:abs_cont},  we conclude that $\|\nu\|=0$ and using Lemma \ref{mlc2}-$(i)$, we have $\nu_{\infty}=0$. Thus, by Lemma \ref{mlc2}-$(ii)$ and Fatou's lemma, we obtain
$$\overline{\lim}_{n \ra \infty} \int_{\Om}|g||\phi_n|^p \dx = \int_{\Om}|g||\phi|^p \dx \leq \underline{\lim}_{n \ra \infty} \int_{\Om}|g||\phi_n|^p \dx \,.$$
Hence, $\lim_{n \ra \infty} \int_{\Om}|g||\phi_n|^p \dx = \int_{\Om}|g||\phi|^p \dx$. This proves that $T_g$ is compact on $\mathcal{D}^{1,p}_{A,V} (\Om).$
\end{proof}
Finally, we provide a way of producing Hardy-weights for which the best constant in \eqref{WH} is attained.  
\begin{proposition}\label{prop-6-10}
 Let $V$ satisfy {\rm{(H1)}}, and let $g_0 \in \H_p(\Om,V)$ be a nonnegative and nonzero function satisfying $\S_{g_0}^{*}=\S_{g_0}^{\infty}=\infty$. Then, for any nonnegative and nonzero $g \in \H_p(\Om,V)$, and $\varepsilon >0$ satisfying $\vge>\S_{g_0}/\S_g$, we have  
	$\S_{g+\varepsilon g_0}< \min\{\S_{g+\varepsilon g_0}^{*},\S_{g+\varepsilon g_0}^{\infty}$\}.

Moreover, if  for such an $\vge$, the potential $g+\varepsilon g_0$ satisfies  {\rm{(H2)}}, then $\S_{g+\varepsilon g_0}$ is achieved at some $\psi \in \mathcal{D}^{1,p}_{A,V^+} (\Om)$. 
\end{proposition}
\begin{remark}\rm
	If $g_0$ and $g$ are  nonzero, nonnegative functions  in $\mathcal{M}^{q}_{\loc}(p;\Om) \cap \H_p(\Om,V)$, and $g_0$ has a compact support in $\Gw$, then  the first part of the theorem holds for $\vge>\S_{g_0}/\S_g$. 
\end{remark}
\begin{proof}[Proof of Proposition~\ref{prop-6-10}]
First, we claim that $\S_g^{*}=\S_{g+g_0}^{*}$. For any $x \in \overline{\Om}$, we have
$$\int_{\Om} (g+g_0) |\phi|^p \dx \leq 
\left[\frac{1}{\S_g(\Om \cap B_r(x))}+ \frac{1}{\S_{g_0}(\Om \cap B_r(x))}\right] 
\!Q_{p,A,V}(\phi) \quad  \forall \phi \in \mathcal{D}^{1,p}_A(\Om \cap B_r(x)) \,.$$
Thus, 
$$\frac{1}{\S_{g+g_0}(\Om \cap B_r(x))} \leq \frac{1}{\S_g(\Om \cap B_r(x))}+ \frac{1}{\S_{g_0}(\Om \cap B_r(x))}.$$ 
By taking $r \ra 0$, we get $\S_{g}(x,\Om) \leq \S_{g+g_0}(x,\Om)$ (as $\S_{g_0}(x,\Om)=\infty$ for all $x$). Obviously,  $\S_{g+g_0}(x,\Om) \leq \S_{g}(x,\Om) $ (as $g,g_0$ are nonnegative). Hence, $\S_{g}(x,\Om) = \S_{g+g_0}(x,\Om)$ for all $x \in \R^N$. Consequently, $\S_{g}^{*} = \S_{g+g_0}^{*}$. Repeating similar  arguments, it follows that $\S_{g}^{\infty} = \S_{g+g_0}^{\infty}$. 

Now, we are ready to prove the proposition. By Theorem \ref{bestconst_thm}, $\S_{g_0}$ is achieved at some $\varphi \in \mathcal{D}^{1,p}_{A,V^+} (\Om)$. For $\varepsilon > \S_{g_0}/\S_g$, 
we have    
$$\S_{g+\varepsilon g_0} \leq \frac{Q_{p,A,V}(\varphi)}{\int_{\Om} (g+\varepsilon g_0)|\varphi|^p \dx} \leq  \frac{Q_{p,A,V}(\varphi)}{\vge\int_{\Om}  g_0|\varphi|^p \dx} = \frac{\S_{g_0}}{\vge} < \S_g .$$
Therefore, 
$$\S_{g+\varepsilon g_0}< \S_g \leq \min\{\S_g^*,\S_g^{\infty}\}= \min\{\S_{g+\varepsilon g_0}^*,\S_{g+ \varepsilon g_0}^{\infty}\}.$$
Moreover,  if $g+\varepsilon g_0$ satisfies  {\rm{(H2)}}, then the second assertion of the proposition follows since $g+\varepsilon g_0$ satisfies the assumptions of Theorem~\ref{bestconst_thm}. 
\end{proof}

 \appendix
 \section{Auxiliary lemmas}\label{appendix1}
The following lemma is an extension of \cite[Lemma~2.10]{DP} to the $(p,A)$-Laplacian case.
 \begin{lemma}\label{weak_lapl}
 	Let $0<u\in W^{1,p}_\loc \cap C(\Omega)$. Let $F\in C^2(\R_+)$ satisfy $F' \geq 0$, and for $p<2$ assume further that $F'(s)^{p-2}F''(s) \to 0$ as $s\to s_0$ where $s_0$ is any critical point of $F$.  Then the following formula holds in the weak sense:
 	\begin{equation}\label{weak_eq}
 	-\Delta_{p,A}(F(u))=-|F'(u)|^{p-2}\left[(p-1)F''(u)|\nabla u|_A^p+F'(u)\Delta_{p,A}(u)\right].
 	\end{equation}
 	Moreover, if $\Delta_{p,A}(u)\in L^1_\loc(\Gw)$, then $\Delta_{p,A}(F(u)) \in L^1_\loc(\Gw)$.
 \end{lemma}
 \begin{proof}
 	Denote $g:=-\Delta_{p,A}(u)$, and let $\varphi\in C_c^{\infty}(\Om)$. By the product and chain rules, we have
 	\begin{multline*}
 	\int_\Omega |\nabla F(u)|_A^{p-2}A\nabla F(u) \!\cdot \! \nabla \varphi\dnu =\\
 	-\!\!\int_\Omega \!|\nabla u|_A^{p\!-\!2}A\nabla u \!\cdot \! \nabla \!\left(|F'(u)|^{p\!-\!2}\!F'(u)\right)\!\varphi\!\dnu
 	+\int_\Omega\!\! |\nabla u|_A^{p\!-\!2}A\nabla u\!\cdot \!\nabla \!\left(|F'(u)|^{p\!-\!2}\!F'(u)\varphi\right)\!\dnu
 	\end{multline*}
 	Note that for $p\geq 2$, the function $\psi(s):=|s|^{p-2}s$ is continuously differentiable, and $\psi'(s):=(p-1)|s|^{p-2}$. Moreover, by our assumption on $F$ it follows that 
 	$|F'(u)|^{p-2}F'(u)\varphi \in W_c^{1,p}(\Gw)$. Consequently, the second term of the right hand side in the above equality equals 
 	$$\int_\Omega |\nabla u|_A^{p-2} A \nabla u\cdot \nabla \left(|F'(u)|^{p-2}F'(u)\varphi\right)\dnu=
 	\int_\Omega g|F'(u)|^{p-2}F'(u)\varphi\dnu.$$
 	Therefore,
 	\begin{multline*}
 	\int_\Omega |\nabla F(u)|_A^{p-2} A \nabla F(u)\cdot \nabla \varphi\dnu =\\
 	 -\int_\Omega |\nabla u|_A^{p-2}A\nabla u\cdot \nabla \left(|F'(u)|^{p-2}F'(u)\right)\varphi \dnu
 	 +\int_\Omega g|F'(u)|^{p-2}F'(u)\varphi\dnu.
 	\end{multline*}
 	Consequently, in the weak sense we have
 	$$-\Delta_{p,A}(F(u))=-|\nabla u|_A^{p-2}A\nabla u\cdot \nabla \left(|F'(u)|^{p-2}F'(u)\right)-\Delta_{p,A}(u)|F'(u)|^{p-2}F'(u).$$
 	Since $\psi'(s)\!:=\!(p-1)|s|^{p-2}$ for $s\!\neq \! 0$, and $\psi'$ is integrable at $0$, we have that in the weak sense
 	\begin{equation*}\label{derive_2}
 	|\nabla u|_A^{p-2}A\nabla u\cdot \nabla \left(|F'(u)|^{p-2}F'(u)\right)=(p-1)|F'(u)|^{p-2} F''(u) |\nabla u|_A^p.
 	\end{equation*}
 	This implies \eqref{weak_eq}, and hence clearly completes the proof of Lemma~\ref{weak_lapl}.
 \end{proof}
\begin{center}
	{\bf Acknowledgments}
\end{center}
The authors thank F.~Gesztesy and  B.~Simon for helpful comments on the literature. The authors acknowledge the support of the Israel Science Foundation (grant 637/19) founded by the
Israel Academy of Sciences and Humanities.

\end{document}